\documentclass[12pt]{article}
\usepackage{amsmath}
\usepackage{amssymb}
\usepackage{amsthm}
\usepackage{latexsym}
\usepackage{esint}
\usepackage{color,fancyhdr,lastpage,graphicx}
\usepackage{subfigure, epsfig, epsf}
\usepackage{xspace}
\usepackage{setspace}
\usepackage{bbm}

\graphicspath{{thesis/Figures/}}

\theoremstyle{plain}

\newtheorem{lemma}{Lemma}[section]

\newtheorem{theorem}[lemma]{Theorem}
\newtheorem{proposition}[lemma]{Proposition}

\theoremstyle{remark}

\def\bb{\begin{color}{black}}
\def\bg{\begin{color}{black}}
\def\br{\begin{color}{black}}
\def\bb{\begin{color}{blue}}
\def\bg{\begin{color}{green}}
\def\br{\begin{color}{red}}
\def\bbr{\begin{color}{brown}}
\def\eg{\end{color}}
\def\er{\end{color}}
\def\eb{\end{color}}
\def\PP{\mathbb{P}}

\def\ve{{\varepsilon}}
\def\le{\leqslant}
\def\pd{\partial}

\def\ge{\geqslant}

\def\es{\emptyset}

\def\E{{\mathbb E}}
\def\O{{\Omega}}

\def\R{{\mathbb R}}

\def\N{{\mathbb N}}

\def\a{{\alpha}}
\def\b{{\beta}}
\def\d{{\delta}}

\def\t{{\tau}}

\def\g{{\gamma}}
\def\G{{\Gamma}}
\def\s{{\sigma}}
\def\l{{\lambda}}

\def\th{{\theta}}

\def\o{{\omega}}
\def\z{{\zeta}}
\def\cA{{\cal A}}

\def\cL{{\cal L}}

\def\cF{{\cal F}}

\def\cH{{\cal H}}

\def\q{\quad}
\def\id{\operatorname{id}}

\def\im{\operatorname{im}}

\def\sgn{\operatorname{sgn}}

\def\dvv{\operatorname{div}}
\def\<{\langle}
\def\>{\rangle}
\def\ra{\rightarrow}

\def\sse{\subseteq}
\def\sm{\setminus}

\renewcommand\epsilon{\ve}
\def\spann{\operatorname{span}}
\def\cinf{C^\infty }

\def\oxy{\Omega^{x,y}}

\def\mexy{\mu^{x,y}_\ve}

\def\mexyU{\mu^{x,y,U}_\ve}

\def\mexyd{\mu^{x,y,\R^d}_\ve}
\def\mex{\mu^x_\ve}

\def\tmexy{\tilde\mu^{x,y}_\ve}

\def\tmexyd{\tilde\mu^{x,y,\R^d}_\ve}
\def\mg{\mu_\g}

\def\dg{\d_\g}

\def\ctm{T^* M}

\def\txm{T_x M}
\def\tym{T_y M}
\def\tgtm{T_{\g_t}M}

\def\tgsm{T_{\g_s}M}

\def\tgoxy{{T_\g\oxy}}
\def\ctxm{T^*_x M}
\def\ctym{T^*_y M}
\def\ctgtm{T^*_{\g_t}M}
\def\ctgsm{T^*_{\g_s}M}

\def\hx{H^x}
\def\hxy{H^{x,y}}

\def\hrm{H^0(\R^m)}

\def\hrp{H^0(\R^p)}
\def\orm{{\O^0(\R^m)}}

\def\tohx{T_\o H^x}
\def\ttohx{\tilde T_\o H^x}
\def\tohxy{T_\o H^{x,y}}
\def\ttohxy{\tilde T_\o H^{x,y}}

\def\tghxy{{T_\g H^{x,y}}}
\def\tghx{{T_\g H^x}}
\def\j{
\def\Ext{\operatorname{Ext}}
\def\C{\operatorname{Cut}}
\def\Hom{\operatorname{Hom}}
\def\Sp{\operatorname{Sp}}
\def\Spin{\operatorname{Spin}}
\def\varpm{\operatorname{pm}}
\def\ord{\operatorname{ord}}
\def\modul{\operatorname{modul}}
\def\sgn{\operatorname{sgn}}
\def\sq{\operatorname{sq}}
\def\supp{\operatorname{supp}}
\def\isom{\operatorname{isom}}
\def\PSL{\operatorname{PSL}}
\def\PIP{\operatorname{PIP}}
\def\PIL{\operatorname{PIL}}
\def\PD{\operatorname{PD}}
\def\PL{\operatorname{PL}}
\def\GL{\operatorname{GL}}
\def\Diff{\operatorname{Diff}}
\def\Out{\operatorname{Out}}
\def\Inn{\operatorname{Inn}}
\def\Aut{\operatorname{Aut}}
\def\End{\operatorname{End}}
\def\Isom{\operatorname{Isom}}
\def\SL{\operatorname{SL}}
\def\id{\operatorname{id}}
\def\Def{\operatorname{Def}}
\def\Trace{\operatorname{Trace}}
\def\trace{\operatorname{trace}}
\def\Tr{\operatorname{Tr}}
\def\Im{\operatorname{Im}}
\def\Re{\operatorname{Re}}
\def\codim{\operatorname{codim}}
\def\Im{\operatorname{Im}}
\def\im{\operatorname{im}}
\def\romancap{\operatorname{cap}}
\def\vol{\operatorname{Vol}}
\def\tan{\operatorname{tan}}
\def\liminf{\operatorname{liminf}}
\def\Lim{\operatorname{Lim}}
\def\mod{\operatorname{mod}}
\def\grad{\operatorname{grad}}
\def\arc{\operatorname{arc}}
\def\loc{\operatorname{loc}}
\def\arctan{\operatorname{arctan}}
\def\cosh{\operatorname{cosh}}
\def\inj{\operatorname{inj}}
\def\deg{\operatorname{deg}}
\def\smear{\operatorname{smear}}
\def\straight{\operatorname{straight}}
\def\support{\operatorname{support}}
\def\represent{\operatorname{represent}}
\def\represents{\operatorname{represents}}
\def\sin{\operatorname{sin}}
\def\id{\operatorname{id}}
\def\volume{\operatorname{Volume}}
\def\wordint{\operatorname{int}}
\def\Div{\operatorname{Div}}
\def\Area{\operatorname{Area}}
\def\area{\operatorname{area}}
\def\diam{\operatorname{diam}}
\def\av{\operatorname{av}}
\def\ex{\operatorname{ex}}
\def\exp{\operatorname{exp}}
\def\curl{\operatorname{curl}}
\def\Curl{\operatorname{Curl}}
\def\Trace{\operatorname{Trace}}
\def\divergence{\operatorname{divergence}}
\def\dim{\operatorname{dim}}
\def\Ker{\operatorname{Ker}}
\def\rank{\operatorname{rank}}
}

\begin{document}

\bibliographystyle{plain}

\begin{center}
\LARGE \textbf{Small-time fluctuations for the bridge of a sub-Riemannian diffusion}

\vspace{0.2in}

\large {\bfseries Ismael Bailleul\footnote{Institut de Recherche Math\'ematiques de Rennes, 263 Avenue du General Leclerc, 35042 Rennes, France},
Laurent Mesnager\footnote{Modal'x, Universit\'e de Paris X, 200 avenue de la R\'epublique, 92001 Nanterre, France}
\&
James Norris\footnote{Statistical Laboratory, Centre for Mathematical Sciences, Wilberforce Road, Cambridge, CB3 0WB, UK}\footnote{Research supported by EPSRC grant EP/103372X/1}
}

\vspace{0.2in}
\small \today

\end{center}
\vspace{0.2in}


\begin{abstract}
We consider small-time asymptotics for diffusion processes conditioned by their initial and final positions, under the assumption that the diffusivity has a sub-Riemannian structure, not necessarily of constant rank. 
We show that, if the endpoints are joined by a unique path of minimal energy, 
and lie outside the sub-Riemannian cut locus,
then the fluctuations of the conditioned diffusion from the minimal energy path, suitably rescaled, converge to a Gaussian limit.
The Gaussian limit is characterized in terms of the bicharacteristic flow, and also in terms of a second variation of the energy functional at the minimal path, the formulation of which is new in this context.
\end{abstract}

\vspace{0.2in}

\section{Introduction}\label{NPD}
Consider a second order differential operator on $\R^d$ in H\"ormander's form\footnote{%
We identify here $X_\ell$ with the differential operator $\sum_{i=1}^dX^i_\ell(x)\pd/\pd x^i$.}
\begin{equation}\label{HOR}
\cL=\frac12\sum_{\ell=1}^mX_\ell^2+X_0
\end{equation}
where $X_0,X_1,\dots,X_m$ are vector fields on $\R^d$.
Let us assume for now that
\begin{equation}\label{BDA}
X_0,X_1,\dots,X_m\text{  are bounded with bounded derivatives of all orders}
\end{equation}
and that $\cL$ satisfies the strong H\"ormander condition on $\R^d$, that is to say,
\begin{equation}\label{LADB}
\spann\{Y(x):Y\in\cA(X_1,\dots,X_m)\}=T_x\R^d,\q\text{for all $x\in\R^d$}.
\end{equation}
Here $\cA(X_1,\dots,X_m)$ denotes the smallest set of vector fields on $\R^d$ containing $X_1,\dots,X_m$ 
and closed under the commutator product, given by
$$
[X,Y](x)=\sum_{i=1}^dX^i(x)\frac{\pd Y}{\pd x^i}(x)-Y^i(x)\frac{\pd X}{\pd x^i}(x).
$$
Our main result concerns the small-time fluctuations of diffusion bridges associated to $\cL$.
For $x,y\in\R^d$, write $\oxy$ for the set of continuous paths $\o:[0,1]\to\R^d$ such that $\o_0=x$ and $\o_1=y$.
For $\ve>0$, denote by $\mexy$ the law on $\oxy$ of the diffusion bridge associated to $\ve\cL$ starting from $x$ at time $0$ and ending at $y$ at time $1$.

Given an absolutely continuous path $\o:[0,1]\to\R^d$, it may be that there exists an absolutely continuous path
$h:[0,1]\to\R^m$ such that, for almost all $t$,
$$
\dot\o_t=\sum_{\ell=1}^mX_\ell(\o_t)\dot h_t^\ell.
$$
Then the energy $I(\o)$ may be defined by
\begin{equation}\label{IP}
I(\o)=\inf\int_0^1|\dot h_t|^2dt
\end{equation}
where the infimum is taken over all such paths $h$. 
If $\o$ is not absolutely continuous, or there is no such path $h$, then we set $I(\o)=\infty$.

In the case where $x$ and $y$ are joined by a unique path $\g$ of minimal energy, we will write $\tgoxy$ for the
set of continuous paths $v:[0,1]\to T\R^d$ such that $v_t\in T_{\g_t}\R^d$ for all $t$ and $v_0=v_1=0$.
Given $\o\in\oxy$ and $\ve>0$, define $\s_\ve(\o)\in\tgoxy$ by
$$
\s_\ve(\o)_t=\frac{\o_t-\g_t}{\sqrt\ve}.
$$
Then define a probability measure $\tmexy$ on $\tgoxy$ by
$$
\tmexy=\mexy\circ\s_\ve^{-1}.
$$
The sub-Riemannian cut locus was defined by Bismut \cite{B} in terms of the bicharacteristic flow associated to the principal symbol $a$ of the operator $2\cL$, which is given by
\begin{equation}\label{PSE}
a(x)=\sum_{\ell=1}^mX_\ell(x)\otimes X_\ell(x). 
\end{equation}
This is reviewed in detail in Section \ref{SRN}.
We can now state a version of our main result.

\begin{theorem}\label{RDMAIN}
Let $\cL$ be a second order differential operator on $\R^d$ of the form \eqref{HOR}.
Assume that $\cL$ satisfies conditions \eqref{BDA} and \eqref{LADB}.
Let $x,y\in\R^d$.
Suppose that there is a unique path $\g$ of minimal energy in $\oxy$ and that $(x,y)$ lies outside the cut locus.
Then $\tmexy$ converges weakly to a Gaussian probability measure $\mu_\g$ on $\tgoxy$ as $\ve\to0$.
\end{theorem}

The covariance of the limit measure $\mu_\g$ is given in terms of the bicharacteristic flow in Section \ref{SRN}.
Theorem \ref{RDMAIN} is proved in Section \ref{LAP}.

This theorem raises some further questions.
First, once the global condition is made that $\g$ have minimal energy, it is natural to hope that a suitable modification of the theorem holds under more local hypotheses.
In particular, we would seek to drop the strong conditions that the underlying space is $\R^d$ and that the operator coefficients are bounded with bounded derivatives of all orders.
%
Second, in the Riemannian case, the limit Gaussian measure $\mu_\g$ can be characterized in terms of the second variation of the energy function at the minimal path.
This leads us to seek an analogous intrinsic object in the sub-Riemannian case.
In order to address these questions, we now reset to a more general framework. 

Let $M$ be a connected $\cinf$ manifold of dimension $d$ and let $a$ be a $\cinf$ non-negative quadratic form on the cotangent space $T^*M$.
We assume that $a$ has a sub-Riemannian structure, that is to say, there exist $m\in\N$ and $\cinf$ vector fields $X_1,\dots,X_m$ on $M$ such that
\begin{equation}\label{LAC}
a(\xi,\xi)=\<\xi,a(x)\xi\>=\sum_{\ell=1}^m\<\xi,X_\ell(x)\>^2,\q \xi\in T_x^*M
\end{equation}
and such that 
\begin{equation}\label{LAD}
\spann\{Y(x):Y\in\cA(X_1,\dots,X_m)\}=T_xM,\q\text{for all $x\in M$}.
\end{equation}
There is associated to the quadratic form $a$ an energy function $I$ on the set of continuous paths $\O=C([0,1],M)$.
While this can be defined as in \eqref{IP}, the following equivalent definition makes clear that $I$ is intrinsic to the quadratic form $a$.
An absolutely continuous path $\o\in\O$ may have a driving path $\xi$, 
by which we mean a measurable path $\xi$ in $T^*M$ such that $\xi_t\in T^*_{\o_t}M$ and $\dot\o_t=a(\o_t)\xi_t$ for almost all $t$. 
Then $\o$ has energy
$$
I(\o)=\int_0^1\<\xi_t,a(\o_t)\xi_t\>dt.
$$
If $\o$ is not absolutely continuous or has no driving path, then $I(\o)=\infty$. 
Write $H^x$ for the subset of $\O$ consisting of paths of finite energy starting from $x$.
For $x,y\in M$, set 
$$
\hxy=\{\o\in H^x:\o_1=y\}.
$$
It is well known, under the bracket condition \eqref{LAD}, that $\hxy$ is non-empty and that the sub-Riemannian distance
\begin{equation}\label{DIST}
d(x,y)=\inf_{\o\in\hxy}\sqrt{I(\o)}
\end{equation}
defines a metric compatible with the topology of $M$.


Our main result concerns the case where $x$ and $y$ are chosen so that $I$ achieves a minimum on $\hxy$ uniquely, say at $\g$.
We will then construct, under a regularity condition on $\g$, a vector space $\tghxy$ of absolutely continuous paths $v$ in $TM$, with $v_t\in\tgtm$ for all $t$ and $v_0=v_1=0$,
along with an equivalence class of norms on $\tghxy$, each making $\tghxy$ into a Hilbert space.
The paths in $\tghxy$ can be thought of as the infinitesimal variations of $\g$ in $\hxy$.
We will further construct a continuous non-negative quadratic form $Q$ on $\tghxy$ such that $Q(v)$ is the minimal second variation of $I$ in the direction $v$, in a sense to be made precise.
These constructions are the content of Section \ref{SUB}.

The sub-Riemannian cut locus, as defined by Bismut \cite{B}, was shown by Ben Arous \cite{BA} to have an alternative characterization 
in terms of a quadratic form on control paths associated to a compatible sub-Riemannian structure.
We will show in Theorem \ref{2.7} that $(x,y)$ lies outside the sub-Riemannian cut locus if and only if our regularity condition on $\g$ holds and $Q$ is positive-definite.
Since the vector space $\tghxy$ and the quadratic form $Q$ are intrinsic to $a$, 
this provides an intrinsic characterization of the cut locus in terms of the energy function in this general setting.

We will show by an example in Section \ref{SUB} that, when $a$ has non-constant rank, it may admit two inequivalent sub-Riemannian structures.
Thus there is a difference in this case between the given data for control theory, namely the sub-Riemannian structure, even up to equivalence, and for hypoelliptic heat flow, namely the quadratic form $a$.

The sub-Riemannian cut locus is less well understood than its Riemannian counterpart. 
It is known to be a closed and symmetric subset of $M\times M$. 
Rifford \& Tr\'elat \cite{MR2139255} and Agrachev \cite{MR2513150} have proved results which limit its size.
The reader may find in the lecture notes of Agrachev, Barilari \& Boscain \cite[Theorems 10.4 and 10.11]{ABB} a proof that, for all $x\in M$ and any $r\in(0,\infty)$ such that $B=\{y\in M:d(x,y)\le r\}$ is compact, 
the set of points $y\in B$ such that $(x,y)$ lies outside the cut locus is dense in $B$.

Let now $\cL$ be a second order differential operator on $M$ with $\cinf$ coefficients, such that $\cL1=0$ and such that $\cL$ has principal symbol $a/2$.
In each coordinate chart, $\cL$ thus takes the form
\begin{equation}\label{LAB}
\cL=\frac12\sum_{i,j=1}^da^{ij}(x)\frac{\pd^2}{\pd x^i\pd x^j}+\sum_{i=1}^db^i(x)\frac{\pd}{\pd x^i}
\end{equation}
for some $\cinf$ functions $b^i$ on $M$.
In this context we refer to $a$ as the diffusivity.
Given that $a$ has sub-Riemannian structure $(X_1,\dots,X_m)$, 
we obtain the same family of operators by choosing another $\cinf$ vector field $X_0$ on $M$ and setting
\begin{equation}\label{HSS}
\cL=\frac12\sum_{\ell=1}^mX_\ell^2+X_0.
\end{equation}
The special case with which we began now corresponds to the following global condition
\begin{equation}\label{URD}
M=\R^d\q\text{and}\q X_0,X_1,\dots,X_m\text{  are bounded with bounded derivative of all orders.}
\end{equation}
We will consider also a different set of conditions, where no further condition is made on $M$ or $a$, in particular no condition of completeness, but we insist that there is a $\cinf$ positive measure $\nu$ and a $\cinf$ $1$-form $\b$ 
such that
\begin{equation}\label{CFO}
\cL f=\tfrac12\dvv(a\nabla f)+a(\b,\nabla f)
\end{equation}
where the divergence is understood with respect to $\nu$, 
and such that $\b$ satisfies the sector condition
\begin{equation}\label{SECT}
\|a(\b,\b)\|_\infty<\infty.
\end{equation}
Note that \eqref{CFO} implies that
$$
X_0=\sum_{\ell=1}^m\a_\ell X_\ell,\q \a_\ell=\frac12\dvv X_\ell+\<\b,X_\ell\>.
$$
In particular, there is a loss with respect to \eqref{URD}, since we now require that, for all $x\in M$,
\begin{equation}\label{HORIZ}
X_0(x)\in\spann\{X_1(x),\dots,X_m(x)\}.
\end{equation}
The estimates obtained by Malliavin calculus under \eqref{URD} make essential use of that condition, and do not appear to allow a variation under local conditions.
On the other hand, heat kernel estimates for incomplete manifolds, on which our localization argument depend, have so far required \eqref{HORIZ}.
We note that the sector condition is satisfied trivially in the case $\b=0$.
Thus our results apply to all sub-Riemannian Laplacians, without any restriction of completeness.

There is a family of probability measures on $\oxy$ which is naturally associated to the operator $\cL$.
Fix $\ve>0$ and $x\in M$. 
There exists a diffusion process starting from $x$ and having generator $\ve\cL$.
Since, in general, the coefficients of $\cL$ may be unbounded and we make no assumption of completeness for $M$, this diffusion may explode with positive probability, that is to say, it may leave all compact sets in a finite time.
We will write $\mex$ for the unique sub-probability measure on $\O$ which is the law of this diffusion restricted to paths which do not explode by time $1$.
Under our assumptions, there is a unique family of probability measures $(\mexy:y\in M)$ on $\O$ which is weakly continuous in $y$, 
with $\mexy$ supported on $\oxy$ for all $y$, and such that
$$
\mex(d\o)=\int_M\mexy(d\o)p(\ve,x,dy)
$$
where $p(\ve,x,.)$ is the (sub-)law of $\o_1$ under $\mex$.
More explicitly, the finite-dimensional distributions of each measure $\mexy$ may be written as follows.
There is a positive $\cinf$ function $p$ on $(0,\infty)\times M\times M$ such that
$$
p(\ve,x,dy)=p(\ve,x,y)\nu(dy).
$$
This function $p$ is the Dirichlet heat kernel for $\cL$ with respect to $\nu$.
Then, for all $k\in \N$, all $t_1,\dots,t_k\in(0,1)$ with $t_1<t_2<\dots<t_k$ and all $x_1,\ldots,x_k\in M$, we have
\begin{align*}
&\mexy(\{\o:\o_{t_1}\in dx_1,\ldots,\o_{t_k}\in dx_k\})\\
&\q\q=\frac{p(\ve t_1,x,x_1)p(\ve(t_2-t_1),x_1,x_2)\ldots p(\ve(1-t_k),x_k,y)}{p(\ve,x,y)}
\nu(dx_1)\dots\nu(dx_k).
\end{align*}
It is straightforward to see that these finite-dimensional distributions are consistent and do not depend on the choice of $\nu$.

Our main aim is to understand the fluctuation behaviour of the diffusion bridge measures $\mexy$ in the limit $\ve\to0$.
A path $\g\in\oxy$ is minimal if $I(\g)<\infty$ and 
$$
I(\g)\le I(\o)\text{ for all $\o\in\oxy$}.
$$
We will say that $\g$ is strongly minimal if, in addition, there exist $\d>0$ and a relatively compact open set $U$ in $M$ such that 
\begin{equation}\label{SMIN}
I(\g)+\d\le I(\o)\text{ for all $\o\in\oxy$ which leave $U$}.
\end{equation}
When $M$ is complete for the sub-Riemannian distance, all metric balls are relatively compact, 
so every minimal path is strongly minimal.
Also, if there is a unique minimal path $\g\in\oxy$, which is strongly minimal, then, by a weak compactness argument,
for all relatively compact domains $U$ containing $\g$, there is a $\d>0$ such that \eqref{SMIN} holds.
\def\j{
\begin{theorem}\label{VAR}
Let $M$ be a connected $\cinf$ manifold.
Let $\cL$ be a second order differential operator on $M$ of the form \eqref{CFO}
$$
\cL f=\tfrac12\dvv(a\nabla f)+a(\b,\nabla f)
$$
where the diffusivity $a$ has a sub-Riemannian structure, 
where the divergence is taken with respect to a positive $\cinf$ locally invariant measure, and where $\b$ is a $\cinf$ $1$-form satisfying the sector condition \eqref{SECT}.
Then, for all $x,y\in M$, as $t\to0$, we have
\begin{equation}\label{LOGP}
t\log p(t,x,y)\to-d(x,y)^2/2.
\end{equation}
Moreover, for all $\d>0$, for $r=\d^{1/4}(d(x,y)^2+\d)^{1/2}$, we have
\begin{equation}\label{LDEV}
\limsup_{\ve\to0}\ve\log\mexy(\{\o\in\oxy:d(\o_t,\G_t(\d))\ge r\text{ for some }t\in[0,1]\})\le-\d/2
\end{equation}
where
$$
\G_t(\d)=\{\g_t:\g\in\hxy,\, I(\g)\le d(x,y)^2+\d\}.
$$
In particular, in the case where there is a unique minimal path $\g\in\hxy$ which is strongly minimal, we have $\mexy\to\dg$ as $\ve\to0$ weakly on $\oxy$, 
where $\dg$ is the unit mass at $\g$.
\end{theorem}
The small-time logarithmic asymptotics \eqref{LOGP} were proved by Varadhan \cite{MR0208191} in the case when $M=\R^d$ and $a$ is everywhere positive-definite.
They were generalized by L\'eandre \cite{MR871256,MR904825} to the sub-Riemannian case, under the global assumption \eqref{URD}.
While the lower bound extends easily, under the hypothesis of Theorem \ref{VAR}, for the upper bound it appears that analytic techniques are needed, and we arrive at \eqref{LOGP} by checking the applicability of a general result of Sturm \cite{MR1355744} in this context.
In the case of Brownian motion in a complete Riemannian manifold, the exponential estimate \eqref{LDEV} follows from a result of Hsu \cite{MR1027823}. 
Under the assumption \eqref{URD} and subject to the condition that $a(x)$ is positive-definite, \eqref{LDEV} follows from work of Inahama \cite{MR3391911}. 

We now go beyond the law of large numbers and exponential asymptotics contained in Theorem \ref{VAR} to identify a small-time equivalent for the heat kernel
and to obtain a central limit theorem for the diffusion bridge measures.
Assume that $x$ and $y$ are chosen so that
\begin{equation}\label{UPM}
\text{there is a unique minimal path $\g\in\hxy$, which is strongly minimal}
\end{equation}
and further that $(x,y)$ lies outside the sub-Riemannian cut locus. 
\br
[Include this now in Theorem \ref{MR}.]
We first state a small-time equivalent for the heat kernel, which was proved by Ben Arous \cite{BA} under the condition \eqref{URD}.
In fact, Ben Arous obtained a full asymptotic expansion, with local regularity statements, but we will be content with the basic form.

\begin{theorem}\label{BALOC}
Under the conditions of Theorem \ref{VAR}, in the case where there is a unique minimal path $\g\in\hxy$, which is strongly minimal, and $(x,y)$ lies outside the cut locus, 
there is a constant $c(x,y)>0$ such that, in the limit $t\to0$,
$$
p(t,x,y)=c(x,y)t^{-d/2}\exp\{-d(x,y)^2/(2t)\}(1+o(1)).
$$
\end{theorem}
\er
}%

For $x,y\in M$, joined by a unique minimal path $\g$, with $(x,y)$ outside the cut locus, the quadratic form $Q$ defines an intrinsic Hilbert norm on $\tghxy$.
Write $\tgoxy$ for the set of continuous paths $v$ in $TM$ such that $v_t\in\tgtm$ for all $t$ and $v_0=v_1=0$.
We make $\tgoxy$ into a Banach space using the uniform norm $\|v\|_\infty=\sup_{t\in[0,1]}|v_t|$ corresponding to a choice of Riemannian metric on $M$.
The associated topology on $\tgoxy$, which is all that matters for us, does not depend on the choice of metric.
We will show in Theorem \ref{SRMR} that there is a unique zero-mean Gaussian measure $\mg$ on $\tgoxy$ such that
$$
\int_\tgoxy\phi(v)^2\mg(dv)=Q(\tilde\phi)
$$
for all continuous linear functionals $\phi$ on $\tgoxy$, where $\tilde\phi\in\tghxy$ is determined by 
$$
\phi(v)=Q(\tilde\phi,v),\q v\in\tghxy.
$$
We rescale the fluctuations of the diffusion bridge around the minimal path $\g$ to obtain a non-degenerate limit.
To do this, we choose a $\cinf$ map $\th:M\to\R^d$ such that the derivative $\th^*(\g_t):\tgtm\to\R^d$ is invertible for all $t\in[0,1]$.
We can always choose a chart $U$ along $\g$ and obtain a suitable function $\th$ by extending the coordinate map from a neighbourhood of $\g$.
Define $\s_\ve:\oxy\to\tgoxy$ by
$$
\s_\ve(\o)_t=\th^*(\g_t)^{-1}(\th(\o_t)-\th(\g_t))/\sqrt\ve.
$$
Then we obtain a probability measure $\tmexy$ on $\tgoxy$ by setting
$$
\tmexy=\mexy\circ\s_\ve^{-1}.
$$
It is straightforward to check that if $\tmexy$ converges weakly on $\tgoxy$ as $\ve\to0$ for one choice of the function $\th$, 
then it does so for all such choices and with the same limit.
Here is our main result.

\begin{theorem}\label{MR}
Let $M$ be a connected $\cinf$ manifold of dimension $d$.
Let $\cL$ be a second order differential operator on $M$ of the form \eqref{CFO} and suppose that the diffusivity of $\cL$ has a
sub-Riemannian structure.
Let $x,y\in M$.
Suppose that there is a unique minimal path $\g\in\hxy$ and that $(x,y)$ lies outside the cut locus.
Suppose, either that
$$
d(x,y)<d(x,\infty)+d(y,\infty)
$$
or that $\g$ is strongly minimal and $\cL$ satisfies the sector condition \eqref{SECT}.
Then there is a constant $c(x,y)>0$ such that, in the limit $t\to0$,
\begin{equation}\label{BALOC}
p(t,x,y)=c(x,y)t^{-d/2}\exp\{-d(x,y)^2/(2t)\}(1+o(1)).
\end{equation}
Moreover, the rescaled diffusion bridge measures satisfy 
$$
\tmexy\to\mg\q\text{ weakly on $\tgoxy$ as $\ve\to0$}.
$$
\end{theorem}

The asymptotic equivalence \eqref{BALOC} is a localized form of the standard small-time equivalent for the heat kernel.
It was shown in the elliptic case, under the condition $d(x,y)<\max\{d(x,\infty),d(y,\infty)\}$ by C. Bellaiche \cite[Chapter 9, Theorem 4.1]{MR634964}.
It was shown in the hypoelliptic case under the condition \eqref{URD} by Ben Arous \cite{BA}.
In the case of Brownian motion on a compact Riemannian manifold, the convergence $\tmexy\to\mg$, with $\mg$ characterized by its covariance, was derived by Molchanov \cite{Mo}, but a full proof was not given.
Many of the techniques we use follow ideas pioneered by Bismut \cite{B} and Ben Arous \cite{BA} in their studies of the heat kernel, indeed, it turns out that the diffusion bridge asymptotics are another side of the same story.
Three recent works have built on the fundamental ideas of Bismut and Ben Arous in complementary directions to ours. 
First, Deuschel, Friz, Jacquier \& Violante \cite{MR3139426,MR3149845} have obtained small time expansions for marginal distributions of the heat flow.
Second, Barilari, Boscain \& Neel \cite{MR3005058} have found estimates of the heat kernel actually on the cut locus.
Third, Habermann \cite{KH} has identified the limiting fluctuations for sub-Riemannian diffusion loops, that is when $x=y$, which can show non-Gaussian behaviour.

The rest of the paper is set out as follows.
In Section \ref{SRN}, we recall the notions of the bicharacteristic flow associated to $a$ and the sub-Riemannian cut locus.
We also give a characterization of the limit fluctuation measure $\mg$ in terms of the bicharacteristic flow.
The proof of Theorem \ref{RDMAIN} is given in Section \ref{LAP}, adapting a method of Azencott, Bismut and Ben Arous, and relying on ideas of Malliavin calculus.
Then in Section \ref{LL}, we show how a heat kernel upper bound provides a localization estimate suitable to deduce Theorem \ref{MR}.
\def\j{
In Section \ref{NFB}, we prove Theorem \ref{VAR} and the estimate \eqref{LOC}.
This section may be read independently of the preceding sections.
}
Sections \ref{SUB} and \ref{QG} are devoted to the geometry of some spaces of paths, leading to the intrinsic construction of the minimal second
variation $Q$ of the energy function and its associated Gaussian measure. 
These two sections may be read independently of the preceding probabilistic and analytic parts.
Finally, in Section \ref{RIEM}, we discuss the case where $a(x)$ is positive-definite for all $x$, and thus defines a Riemannian metric. 
In this case the limit measure $\mg$ has several different characterizations using classical objects of Riemannian geometry. 

We are grateful to Emmanuel Tr\'elat and Fabrice Baudoin for helpful advice in the course of this work, and to Karen Habermann for a careful reading of the manuscript.

\section{Bicharacteristic flow, cut locus and fluctuation measure}\label{SRN}
The discussion in this section applies both, in the case $M=\R^d$, to the principal symbol $a$ defined at \eqref{PSE}
and to the quadratic form $a$ introduced at \eqref{LAC}. 
The bicharacteristic flow is the maximal flow $(\psi_t(\l):\l\in T^*M,t\in D(\l))$ of the vector field $V$ on $T^*M$
given by 
$$
\b(V,.)=d\cH
$$ 
where $\b$ is the canonical symplectic $2$-form on $T^*M$ and $\cH:T^*M\to[0,\infty)$ is the Hamiltonian
$$
\cH(\l)=\tfrac12\<\l,a(x)\l\>,\q \l\in T_x^*M.
$$
Thus, for each $\l\in T^*M$, $D(\l)=(\z^-,\z^+)$ is an open interval containing $0$, we have $\psi_0(\l)=\l$ and 
$$
\dot\psi_t(\l)=V(\psi_t(\l)),\q t\in D(\l)
$$
and $\psi_t(\l)$ leaves all compact sets in $T^*M$ as $t\to\z^+$ if $\z^+<\infty$ and as $t\to\z^-$ if $\z^->-\infty$.
If $M$ is complete for the sub-Riemannian metric, then $D(\l)=\R$ for all $\l\in\ctm$.
In a coordinate chart $U$ for $M$, for $\l\in U$, write $\l_t=\psi_t(\l)=(x_t,p_t)$ while $\l_t$ remains in $T^*U$, with $x_t\in U$ and $p_t\in\R^d$.
Then, for any sub-Riemannian structure $(X_1,\dots,X_m)$ for $a$, 
writing $\<p,\nabla X(x)\>_i$ for $(\pd/\pd x^i)\<p,X(x)\>$, we have
\begin{align}\notag
\dot x_t&=\sum_{\ell=1}^m\<p_t,X_\ell(x_t)\>X_\ell(x_t),\q x_0=x\\
\label{BCF}
\dot p_t&=-\sum_{\ell=1}^m\<p_t,X_\ell(x_t)\>\<p_t,\nabla X_\ell(x_t)\>,\q p_0=p.
\end{align}

Fix $x,y\in M$ and suppose that 
\begin{equation}\label{UMP}
\text{there is a unique minimal path $\g$ in $\hxy$}.
\end{equation}
Write $\pi$ for the projection $T^*M\to M$.
We assume that $\g$ is a normal minimizer, meaning that
\begin{equation}\label{PBC}
\text{there exists a bicharacteristic $(\l_t)_{t\in[0,1]}$ such that $\g_t=\pi\l_t$ for all $t\in[0,1]$.}
\end{equation}
For $t\in[0,1]$, write $\l_t=\psi_t(\l_0)$ and define linear maps $J_t:T^*_xM\to\tgtm$ and $K_t:T^*_yM\to\tgtm$ by
\begin{equation}\label{JFE}
J_t\xi_0
=\left.\frac{\pd}{\pd\ve}\right|_{\ve=0}\pi\psi_t(\l_0+\ve\xi_0),\quad
K_t\xi_1
=\left.\frac{\pd}{\pd\ve}\right|_{\ve=0}\pi\psi_{-(1-t)}(\l_1-\ve\xi_1).
\end{equation}
We assume that $(x,y)$ is non-conjugate for $(\l_t)_{t\in[0,1]}$, meaning that 
\begin{equation}\label{NC}
\text{$J_1$ is invertible.}
\end{equation}
Following Bismut \cite{B}, when conditions \eqref{UMP}, \eqref{PBC} and \eqref{NC} hold, the pair $(x,y)$ is said to lie outside the cut locus of $a$.
By a simple and well known argument, \eqref{NC} implies that the bicharacteristic projecting to $\g$ is unique.

The following statement is part of Theorem \ref{SRMR}.
\begin{theorem}\label{MRM}
Assume the hypotheses of Theorem \ref{RDMAIN} or those of Theorem \ref{MR}. 
Then the limit measure $\mg$ is the unique zero-mean Gaussian measure on $\tgoxy$ with covariance given for $s\le t$ by
\begin{equation}\label{ZMG}
\int_\tgoxy v_s\otimes v_t\,\mu_\g(dv)=J_sJ_1^{-1}K_t^*.
\end{equation}
\end{theorem}
The characterization of the cut locus and limit measure $\mg$ in terms of the bicharacteristic flow is computationally effective and does not require the
construction of the quadratic form $Q$.
On the other hand, the alternative characterizations given Theorems \ref{2.7} and \ref{SRMR}, in terms of $Q$, confirm in the setting of an infinite-dimensional path space 
an intuition derived from analogous considerations for functions and measures in finite dimensions.
Moreover, the proof of Theorem \ref{RDMAIN} is by analysis in path space and leads naturally to the formulation given.

We conclude this section with some remarks on symmetry under time-reversal.
The following calculation shows that $J_1=K_0^*$ and hence that the cut locus is symmetric.
Since $V$ is Hamiltonian, its flow preserves the symplectic form $\b$.
See for example \cite{M}.
For $\xi\in\ctm$, write $\tilde\xi$ for the corresponding vertical vector in $T\ctm$ and write $\psi_t^*$ for the action of $\psi_t$ on $T\ctm$. 
Then
$$
\<J_1\xi_0,\xi_1\>
=\<\pi^*\psi_1^*\tilde\xi_0,\xi_1\>
=\b(\psi_1^*\tilde\xi_0,\tilde\xi_1)
=\b(\tilde\xi_0,\psi_{-1}^*\tilde\xi_1)
=-\<\xi_0,\pi^*\psi_{-1}^*\tilde\xi_1\>
=\<\xi_0,K_0\xi_1\>.
$$
Consider the case where $\cL$ has the form \eqref{CFO} and $\dvv(a\b)=0$, that is, $\nu$ is invariant for $\cL$.
Then, for compactly supported $\cinf$ functions $f,g$ on $M$, we have
$$
\int_Mf\cL g\,d\nu=\int_Mg\hat\cL f\,d\nu.
$$
where 
$$
\hat\cL f=\tfrac12\dvv(a\nabla f)-a(\b,\nabla f).
$$
and the associated bridge measure satisfies $\hat\mu_\ve^{y,x}=\mexy\circ{\wedge}^{-1}$,
where $\wedge:\oxy\to\O^{y,x}$ is the time-reversal map, given by $\hat\o_t=\o_{1-t}$.
Hence, in this case, the whole set-up is invariant under time-reversal.

\section{Laplace's method on Wiener space}\label{LAP}
In this section we prove Theorem \ref{RDMAIN}, following closely the method used by Ben Arous \cite{BA} to study the heat kernel.
The vector fields $X_0,X_1,\dots,X_m$ provide both a means to construct a diffusion process $(x^\ve_t)_{t\ge0}$ with generator $\ve\cL$ starting from $x$ 
and a parametrization of the set $H^x$ of finite energy paths starting from $x$. 
Let $(B_t)_{t\in[0,1]}$ be a Brownian motion in $\R^m$, which we assume is realized as the coordinate process on $\O^0(\R^m)=\{w\in C([0,1],\R^m):w_0=0\}$ under Wiener measure $\PP$.
Define a vector field $\tilde X_0$ on $\R^d$ by
$$
\tilde X_0^i(x)=X_0^i(x)+\frac12\sum_{\ell=1}^m\<\nabla X^i_\ell(x),X_\ell(x)\>.
$$
Consider the It\^o stochastic differential equation in $\R^d$
$$
dx_t^\ve=\sqrt\ve\sum_{\ell=1}^mX_\ell(x_t^\ve)dB_t^\ell+\ve\tilde X_0(x_t^\ve)dt,\q x_0^\ve=x.
$$
This has a unique strong solution $(x_t^\ve)_{t\in[0,1]}$, whose law on $\O=C([0,1],\R^d)$ is $\mu_\ve^x$.
There is no explosion.
Write $H^0(\R^m)$ for the set of Cameron--Martin paths $h$ in $\R^m$, that is to say, the set of absolutely continuous functions $h:[0,1]\to\R^m$, starting from $0$, such that
$$
\int_0^1|\dot h_t|^2dt<\infty.
$$
Given $h\in H^0(\R^m)$, consider the differential equation in $\R^d$
\begin{equation}\label{FEP}
d\phi_t=\sum_{\ell=1}^mX_\ell(\phi_t)dh^\ell_t,\q \phi_0=x.
\end{equation}
There is a unique solution $\phi(x,h)=(\phi_t(x,h))_{t\in[0,1]}$ in $\O$.
As we will show in Proposition \ref{2.1}, in fact $\phi(x,h)\in H^x$ and $\phi(x,.)$ maps $H^0(\R^m)$ onto $H^x$.
Recall that $\g$ is the minimizing path in $H^{x,y}$ and $\g$ is the projection of a bicharacteristic $\l$.
Define $h\in H^0(\R^m)$ by $\dot h_t^\ell=\<\l_t,X_\ell(\g_t)\>$.
Then $\g=\phi(x,h)$ and $I(\g)=\int_0^1|\dot h_t|^2dt$.
We reserve the notation $h$ for this minimizing control path from now on.
For $s\in[0,1]$, we will write $(\phi_{ts}(x,h))_{t\in[s,1]}$ for the solution to \eqref{FEP} starting from $x$ at time $s$.
We denote the derivative in $x$ by $\phi^*_{ts}(x,h)$ and set $u_t=\phi_t^*(x,h)$.
Then $u_t$ is invertible for all $t$ and 
$$
du_t=\sum_{\ell=1}^m\nabla X_\ell(\g_t)u_tdh^\ell_t,\q u_0=I.
$$
Moreover, we have $\phi^*_{ts}(\g_s,h)=u_tu_s^{-1}$.
We can define a continuous linear map $v:\O^0(\R^m)\to\tgoxy$ by\footnote{%
Since $M=\R^d$ in the present discussion, $\tgoxy$ can be naturally identified with a subset of $\O$, but we will keep the distinction anyway.}
\begin{equation}\label{VEQ}
v_t(w)=\sum_{\ell=1}^m\int_0^t\phi^*_{ts}(\g_s,h)X_\ell(\g_s)dw_s^\ell=\sum_{\ell=1}^mu_t\int_0^tu_s^{-1}X_\ell(\g_s)dw_s^\ell
\end{equation}
where the integral is understood by a formal integration by parts.
Set $Y_t=v_t(B)$ and note that $Y_1$ is a zero-mean Gaussian random variable in $\R^d$ having covariance matrix $\bar C_1=u_1C_1u_1^*$, where
$$
C_1=\sum_{\ell=1}^m\int_0^1(u_t^{-1}X_\ell(\g_t))\otimes(u_t^{-1}X_\ell(\g_t))dt.
$$
In \cite{B}, Bismut called $C_1$ the deterministic Malliavin covariance matrix.
Condition \eqref{NC} implies that $C_1$ is invertible. 
This follows in particular from \eqref{al}.
Hence $Y_1$ has a density function with respect to Lebesgue measure on $\R^d$, given by
$$
\bar p(z)=(2\pi)^{-d/2}(\det\bar C_1)^{-1/2}\exp\{-\<z,\bar C_1^{-1}z\>/2\}.
$$
Set\footnote{%
Formally, we have
$$
Y_t=\frac\pd{\pd h}\phi_t(x,h)(B),\q Y^{(2)}_1=\frac{\pd^2}{\pd h^2}\phi_1(x,h)(B,B)
$$
which may be seen by differentiating \eqref{FEP} twice in $h$ in the direction $B$ and solving the resulting equations by variation of constants.
}
$$
Y^{(2)}_1=\sum_{\ell=1}^m\int_0^1\phi^*_{1t}(\g_t,h)\left\{\nabla^2X_\ell(\g_t)(Y_t,Y_t)dh_t^\ell+2\nabla X_\ell(\g_t)Y_tdB_t^\ell\right\},\q S=\<\l_1,Y_1^{(2)}\>.
$$
where the integral $Y_tdB_t$ is understood in the sense of It\^o.
Define a linear map $\t:\R^d\to\hrm$ by
$$
\dot\t_t^\ell(z)=\left\<C_1^{-1}u_1^{-1}z,u_t^{-1}X_\ell(\g_t)\right\>.
$$
Set $K=\{k\in\hrm:v_1(k)=0\}$.
From \eqref{VEQ} we see that $\t$ maps $\R^d$ onto the orthogonal complement $K^\perp$ of $K$ in $\hrm$. 
Moreover, $v_1(\t(z))=z$ and the restriction of the map $\t\circ v_1$ to $\hrm$ is the orthogonal projection $\hrm\to K^\perp$.
Set
\begin{equation}\label{BWW}
W_t=B_t-W'_t,\q W'_t=\t_t(v_1(B)).
\end{equation}
Then $(W_t)_{t\in[0,1]}$ and $(W'_t)_{t\in[0,1]}$ are independent continuous Gaussian processes, and $v_1(W)=0$ and $W'\in K^\perp$ almost surely.
For $z\in\R^d$, set $W_t(z)=W_t+\t_t(z)$ and $Y_t(z)=v_t(W(z))$ and $S(z)=\<\l_1,Y_1^{(2)}(z)\>$, where
\begin{equation}\label{SDEF}
Y^{(2)}_1(z)=\sum_{\ell=1}^m\int_0^1\phi^*_{1t}(\g_t,h)\left\{\nabla^2X_\ell(\g_t)(Y_t(z),Y_t(z))dh_t^\ell+2\nabla X_\ell(\g_t)Y_t(z)dW_t^\ell(z)\right\}.
\end{equation}
We interpret the integral with respect to $W(z)$ by writing 
$$
Y_t(z)dW_t(z)=v_t(B-W'+\t(z))d(B_t-W_t'+\t_t(z))
$$ 
and expanding. 
The term $v_t(B)dB_t$ is considered as an It\^o integral, 
while the remaining terms can be considered as integrals with respect to Lebesgue measure, sometimes after a formal integration by parts.
It is straightforward to see, using the same version of the It\^o integral for all $z$, 
that the family of random variables $(Y^{(2)}_1(z):z\in\R^d)$ is continuous in $z$.
Also, $Y(z)$ and $Y_1^{(2)}(z)$ are independent of $v_1(B)$ for all $z$.
Note that $W(v_1(B))=B$ and $Y(v_1(B))=Y$, and that $S(v_1(B))=S$ almost surely. 
Note also that $Y^{(2)}_1(z)$ belongs to the sum of the zeroth, first and second Wiener chaoses in $L^2(\O^0(\R^m),\PP)$ for all $z$.
By Theorem \ref{SRMR}, we have $\E(e^{pS(z)/2})<\infty$ for all $z\in\R^d$ for some $p>1$.

Consider the function 
$$
E(z)=d(x,z)^2/2=\inf_{\o\in H^{x,z}}I(\o)/2.
$$ 
As Bismut \cite[Theorem 1.26]{B} showed, $E$ is $\cinf$ in a neighbourhood $N$ of $y$, with $E'(y)=\l_1$.
Following Ben Arous \cite[Lemma 3.8]{BA}, there then exists a function $F\in C^\infty_b(\R^d)$ such that the map $F+E$ has a unique and non-degenerate 
minimum at $y$, with minimum value $0$.
To see this, choose a neighbourhood $N_0$ of $y$ and a $C^\infty$ function $\chi$ of compact support such that $1_{N_0}\le\chi\le1_N$.
Fix a constant $\a>0$ and consider the function
$$
F(z)=\chi(z)\left(\a|y-z|^2-E(y)-E'(y)(z-y)\right)+(1-\chi(z)).
$$
Then $F\in C^\infty_b(\R^d)$ and $F(y)+E(y)=0$ and $F'(y)+E'(y)=0$.
Moreover, by choosing $\a$ sufficiently large, we can ensure that $F''(y)+E''(y)$ is positive-definite and $F(z)+E(z)>0$ for all $z\not=y$,
so $F+E$ has a non-degenerate minimum at $y$, which is also its global minimum.
We fix a choice of a function $F$ with the given properties for the rest of the analysis.

Set $\g^0=\g$ and define $(\g_t^\ve)_{t\in[0,1]}$ for $\ve>0$ as the strong solution of the stochastic differential equation
\begin{equation}\label{GSDE}
d\g_t^\ve=\sum_{\ell=1}^mX_\ell(\g_t^\ve)dh_t^\ell+\sqrt\ve\sum_{\ell=1}^mX_\ell(\g_t^\ve)dB_t^\ell+\ve\tilde X_0(\g_t^\ve)dt,\q \g_0^\ve=x.
\end{equation}
By standard results on stochastic differential equations, for all $t\in[0,1]$ and all $p\in[1,\infty)$, the map $\ve\mapsto\g_t^\ve:[0,\infty)\to L^p(\PP)$ is continuous.
Furthermore, we can and do choose versions so that, almost surely, the map $\s\mapsto\g_t^{\s^2}:[0,\infty)\to\R^d$ is $\cinf$.
Moreover, the first and second derivatives at $\s=0$ then satisfy
$$
\left.\frac\pd{\pd\s}\right|_{\s=0}\g_t^{\s^2}=Y_t,\q 
\left.\left(\frac\pd{\pd\s}\right)^2\right|_{\s=0}\g_1^{\s^2}=Y^{(2)}_1+Z_1,\q
Z_1=\int_0^1\phi_{1t}^*(\g_t,h)\tilde X_0(\g_t)dt.
$$
Now the map $f(\s)=F(\g_1^{\s^2})$ is $\cinf$ on $[0,\infty)$ and $F(y)=-d(x,y)^2/2$ and $F'(y)=-\l_1$, so
$$
f(0)=F(y)=-\frac12\int_0^1|\dot h_t|^2dt,\q
f'(0)=F'(y)Y_1=-\sum_{\ell=1}^m\int_0^1\dot h_t^\ell dB_t^\ell
$$
and
$$
f''(0)=F'(y)(Y^{(2)}_1+Z_1)+F''(y)(Y_1,Y_1).
$$
Set
$$
R(\ve)=\int_0^1(1-\th)f''(\th\sqrt\ve)d\th.
$$
Then, by Taylor's theorem,
$$
F(\g_1^\ve)=f(\sqrt\ve)=f(0)+\sqrt\ve f'(0)+\ve R(\ve)=-\frac12\int_0^1|\dot h_t|^2dt-\sqrt\ve\sum_{\ell=1}^m\int_0^1\dot h_t^\ell dB_t^\ell+\ve R(\ve).
$$
Set 
$$
\tilde x^\ve_t=\frac{x_t^\ve-\g_t}{\sqrt\ve},\q\tilde\g_t^\ve=\frac{\g^\ve_t-\g_t}{\sqrt\ve}
$$
and note that, for all $t\in[0,1]$, we have $\tilde\g^\ve_t\to Y_t$ as $\ve\to0$ almost surely. 
For $\ve>0$, consider the new probability measure $\PP^\ve$ on $\orm$ given by $d\PP^\ve/d\PP=\rho_1^\ve$, where
$$
\rho_1^\ve=\exp\left\{-\frac1{\sqrt\ve}\sum_{\ell=1}^m\int_0^1\dot h_t^\ell dB_t^\ell-\frac1{2\ve}\int_0^1|\dot h_t|^2dt\right\}=\exp\left\{\frac{F(\g_1^\ve)}\ve-R(\ve)\right\}.
$$
Note that \eqref{GSDE} can be rewritten in the form
$$
d\g_t^\ve=\sqrt\ve\sum_{\ell=1}^mX_\ell(\g_t^\ve)dB_t^{\ve,\ell}+\ve\tilde X_0(\g_t^\ve)dt,\q \g_0^\ve=x
$$
where $B^{\ve,\ell}_t=B^\ell_t+h_t^\ell/\sqrt\ve$.
By the Cameron-Martin formula, under $\PP^\ve$, the process $(B^\ve_t)_{t\in[0,1]}$ is a Brownian motion, so $\g^\ve$ has law $\mu_\ve^x$.

At this point we modify the argument of Ben Arous by introducing a smooth cylindrical function $G$ on $\O$ which serves to keep track of the paths of the diffusion bridge.
Fix $t_1,\dots,t_k\in(0,1)$ with $t_1<\dots<t_k$ and a $\cinf$ function $g$ on $(\R^d)^k$ of polynomial growth.
Set $G(\o)=g(\o_{t_1},\dots,\o_{t_k})$.
Define for $z\in\R^d$ and $\ve>0$
\begin{align*}
G_\ve(z)
&=\ve^{d/2}p(\ve,x,y+\sqrt\ve z)e^{-F(y+\sqrt\ve z)/\ve}\int_{C([0,1],\R^d)} G(\o)\tilde\mu_\ve^{x,y+\sqrt\ve z}(d\o),\\
G_0(z)
&=\bar p(z)e^{\<\l_1,Z_1\>/2-F''(y)(z,z)/2}\E(G(Y(z))e^{S(z)/2}).
\end{align*}
Here we have written $\tilde\mu_\ve^{x,y+\sqrt\ve z}$ for the law of $(\o-\g)/{\sqrt\ve}$ when $\o$ has law $\mu_\ve^{x,y+\sqrt\ve z}$.
Then $G_\ve$ and $G_0$ are continuous integrable functions on $\R^d$.
Consider the Fourier transform
$$
\hat G_\ve(\xi)=\int_{\R^d}G_\ve(z)e^{i\<\xi,z\>}dz.
$$
Then
\begin{align}\notag
\hat G_\ve(\xi)
&=\int_{\R^d}\int_\O p(\ve,x,y')e^{-F(y')/\ve}G\left(\frac{\o-\g}{\sqrt\ve}\right)\mu_\ve^{x,y'}(d\o)e^{i\<\xi,(y'-y)/\sqrt\ve\>}dy'\\\notag
&=\int_\O e^{-F(\o_1)/\ve}G\left(\frac{\o-\g}{\sqrt\ve}\right)e^{i\<\xi,(\o_1-\g_1)/\sqrt\ve\>}\mu_\ve^x(d\o)\\\notag
&=\E\left(G(\tilde x^\ve)\exp\{i\<\xi,\tilde x_1^\ve\>-F(x_1^\ve)/\ve\}\right)\\\notag
&=\E\left(G(\tilde\g^\ve)\exp\{i\<\xi,\tilde \g_1^\ve\>-F(\g_1^\ve)/\ve\}\rho_1^\ve\right)\\\label{GEP}
&=\E\left(G(\tilde\g^\ve)\exp\{i\<\xi,\tilde \g_1^\ve\>-R(\ve)\}\right)
\end{align}
and
\begin{align*}
\hat G_0(\xi)
&=\int_{\R^d}\bar p(z)e^{\<\l_1,Z_1\>/2-F''(y)(z,z)/2}\E(G(Y(z))e^{S(z)/2})e^{i\<\xi,z\>}dz\\
&=\E\left(G(Y)\exp\{i\<\xi,Y_1\>+(S+\<\l_1,Z_1\>-F''(y)(Y_1,Y_1))/2\}\right)\\
&=\E\left(G(Y)\exp\{i\<\xi,Y_1\>-f''(0)/2\}\right).
\end{align*}
Consider the limit $\ve\to0$.
We have $\tilde\g^\ve_t\to Y_t$ almost surely and in $L^p$ for all $p<\infty$, for all $t\in[0,1]$.
Now $R(\ve)\to f''(0)/2$ almost surely as $\ve\to0$ and we have the following key estimate proved by Ben Arous \cite[Lemma 3.25]{BA}: there exists $p>1$ such that
\begin{equation}\label{BKE}
\limsup_{\ve\to0}\E(e^{-pR(\ve)})<\infty.
\end{equation}
Hence $\hat G_\ve(\xi)\to\hat G_0(\xi)$ for all $\xi\in\R^d$.
We will prove the following key lemma at the end of this section.
\begin{lemma}\label{KL}
There exists a constant $C(G)<\infty$ such that, for all $\ve\in(0,1]$ and all $\xi\in\R^d$, we have
\begin{equation}\label{HRR}
|\hat G_\ve(\xi)|\le C(G)/(1+|\xi|^{d+1}).
\end{equation}
Moreover, there exists $C<\infty$ such that, uniformly in $s,t\in[0,1]$, in the case where $G(\o)=|\o_s-\o_t|^4$, 
for all $\ve\in(0,1]$ and all $\xi\in\R^d$, we have
\begin{equation}\label{GST}
|\hat G_\ve(\xi)|\le C|s-t|^2/(1+|\xi|^{d+1}).
\end{equation}
\end{lemma}
The lemma allows us to use dominated convergence in the Fourier inversion formula to deduce that
$$
G_\ve(0)=(2\pi)^{-d}\int_{\R^d}\hat G_\ve(\xi)d\xi\to(2\pi)^{-d}\int_{\R^d}\hat G_0(\xi)d\xi=G_0(0).
$$
that is to say
\begin{equation}\label{KF}
\ve^{d/2}p(\ve,x,y)e^{d(x,y)^2/(2\ve)}\int_\tgoxy G(\o)\tilde\mu_\ve^{x,y}(d\o)\to(2\pi)^{-d/2}(\det\bar C_1)^{-1/2}e^{\<\l_1,Z_1\>/2}\E(G(Y(0))e^{S(0)/2})
\end{equation}
where we used the fact that $\bar p(0)=(2\pi)^{-d/2}(\det\bar C_1)^{-1/2}$.
On taking $g=1$, we recover the heat kernel asymptotics shown by Ben Arous
\begin{equation}\label{BBF}
\ve^{d/2}p(\ve,x,y)e^{d(x,y)^2/(2\ve)}\to(2\pi)^{-d/2}(\det\bar C_1)^{-1/2}e^{\<\l_1,Z_1\>/2}\E(e^{S(0)/2}).
\end{equation}
By Theorem \ref{SRMR}, $\mg$ is the law of $Y(0)$ on $\tgoxy$ under the probability measure $\tilde\PP$, where $d\tilde\PP/d\PP\propto e^{S(0)/2}$. 
Hence, on dividing \eqref{KF} by \eqref{BBF}, we obtain
$$
\int_\tgoxy G(\o)\tmexy(d\o)\to\int_\tgoxy G(\o)\mg(d\o)
$$
from which we deduce that the finite-dimensional distributions of $\tmexy$ converge weakly to those of $\mu_\g$.
Finally, on taking $G(\o)=|\o_s-\o_t|^4$ and using the estimate \eqref{GST}, we find by Fourier inversion that 
$$
\ve^{d/2}p(\ve,x,y)e^{d(x,y)^2/(2\ve)}\int_\tgoxy|\o_s-\o_t|^4\tmexy(d\o)=G_\ve(0)\le C|s-t|^2.
$$
Since the right-hand side in \eqref{BBF} is positive, we deduce that for some $C<\infty$ we have, for all $s,t\in[0,1]$,
$$
\sup_{\ve\in(0,1]}\int_\tgoxy|\o_s-\o_t|^4\tmexy(d\o)\le C|s-t|^2.
$$
Hence, by standard arguments, the family of laws $(\tmexy:\ve\in(0,1])$ is tight on $\tgoxy$ and so $\tmexy\to\mg$ weakly as $\ve\to0$.
In particular, we also have $\mexy\to\dg$, so the proof is complete.

\begin{proof}[Proof of Lemma \ref{KL}]
The idea is to integrate by parts in \eqref{GEP}, $d+1$ times, using Malliavin calculus. 
We will vary the argument of \cite{BA} in three respects.
First, we will present the argument using Bismut's integration by parts formula \cite{MR621660,MR942019} for solutions of stochastic differential equations. 
This is more direct and may be followed in detail without knowledge of the general apparatus of Malliavin calculus. 
Second, we will use the corrected argument \cite{MES} for the uniform non-degeneracy of the Malliavin covariance matrix.
Third, we will include the simple modifications needed to go beyond the case $G=1$ which is covered in \cite{BA}.
Within the proof, we will use a few notations which conflict with usage elsewhere.

Fix $\s=\sqrt\ve>0$.
Define processes $(x_t)_{t\in[0,1]}$ in $\R^d$ and $(u_t)_{t\in[0,1]}$, $(v_t)_{t\in[0,1]}$ in $\R^d\otimes(\R^d)^*$ as the strong
solutions of the following system of stochastic differential equations
\begin{align}\label{DGE}
dx_t
&=\sum_{\ell=1}^mX_\ell(x_t)dh^\ell_t+\sum_{\ell=1}^m\s X_\ell(x_t)dB^\ell_t+\s^2\tilde X_0(x_t)dt, \q x_0=x\\\notag
du_t
&=\sum_{\ell=1}^m\nabla X_\ell(x_t)u_tdh^\ell_t+\sum_{\ell=1}^m\s\nabla X_\ell(x_t)u_tdB^\ell_t
+\s^2\nabla\tilde X_0(x_t)u_tdt, \q u_0=I\\
\label{VGE}
dv_t
&=-\sum_{\ell=1}^mv_t\nabla X_\ell(x_t)dh^\ell_t-\sum_{\ell=1}^m\s v_t\nabla X_\ell(x_t)dB^\ell_t
-\s^2v_t\left\{\nabla\tilde X_0-\sum_{\ell=1}^m(\nabla X_\ell)^2\right\}(x_t)dt, \q v_0=I.
\end{align}
Then $x_t=\g^{\s^2}_t$ and, by It\^o's formula, $v_t=u_t^{-1}$ for all $t$.
It is well known that the random variables $\sup_{t\in[0,1]}|x_t|$, $\sup_{t\in[0,1]}|u_t|$ and $\sup_{t\in[0,1]}|v_t|$ have moments of all orders, 
which are bounded uniformly in $\s\in[0,1]$.

Consider the following stochastic differential equation in $\R^N$ with $\cinf$ coefficients
\begin{equation}\label{ZT}
dz_t=\sum_{\ell=1}^mW_\ell(z_t)dh_t^\ell+\sum_{\ell=1}^mZ_\ell(z_t)dB^\ell_t+Z_0(z_t)dt,\q z_0=z.
\end{equation}
We assume that the coefficients have a graded Lipschitz structure.
By this we mean that the coefficients and all their derivatives satisfy polynomial growth bounds, 
and that there exist $k\in\N$ and $N_1,\dots,N_k\in\N$ and a decomposition $z_t=(z_t^1,\dots,z_t^k)$ such that $N_1+\dots+N_k=N$,
and, for $j=1,\dots,k$, the component $z_t^j$ takes values in $\R^{N_j}$ and, for all $\ell$, 
the corresponding components $W_\ell^j$ and $Z_\ell^j$ of the coefficients depend only on $(z^1,\dots,z^j)$,
with $\pd W_\ell^j/\pd z^j$ and $\pd Z_\ell^j/\pd z^j$ uniformly bounded.
We will allow the coefficients in \eqref{ZT} to depend measurably on $\s$ and $t$, without making this explicit in the notation, 
while assuming that the bounds in their graded Lipschitz structure hold uniformly in $\s\in[0,1]$ and $t\in[0,1]$.
This is sufficient to ensure the existence and uniqueness of a strong solution $(z_t)_{t\in[0,1]}$ to \eqref{ZT} 
such that $\sup_{t\in[0,1]}|z_t|$ has moments of all orders, bounded uniformly in $\s\in[0,1]$.
See \cite{MR942019} for more details.

Fix $i\in\{1,\dots,d\}$ and consider for $\eta\in\R$ a perturbed process $(B_t^\eta)_{t\in[0,1]}$ in $\R^m$ given by
$$
dB_t^{\eta,\ell}=dB_t^\ell+\eta(v_tX_\ell(x_t))^idt,\q B_0^\eta=0.
$$
Write $(z_t^\eta)_{t\in[0,1]}$ for the strong solution of the stochastic differential equation which is obtained when we replace $(B_t)_{t\in[0,1]}$ 
in \eqref{ZT} by this perturbed process.
We can and do choose a version of the family of processes $(z_t^\eta)_{t\in[0,1]}$ which is almost surely $\cinf$ in $\eta$.
We associate to $(z_t)_{t\in[0,1]}$ the derived process $(z_t')_{t\in[0,1]}=((z_t')^1,\dots,(z_t')^d)_{t\in[0,1]}$ in $\R^N\otimes\R^d$ given by 
\begin{equation}\label{DDE}
(z_t')^i=(\pd/\pd\eta)|_{\eta=0}z_t^\eta.
\end{equation}
Then $(z'_t)_{t\in[0,1]}$ satisfies the following stochastic differential equation in $\R^N\otimes\R^d$
\begin{equation}\label{ZPT}
dz'_t=\sum_{\ell=1}^m\nabla W_\ell(z_t)z_t'dh_t^\ell+\sum_{\ell=1}^m\nabla Z_\ell(z_t)z_t'dB^\ell_t+\nabla Z_0(z_t)z'_tdt
+\sum_{\ell=1}^mZ_\ell(z_t)\otimes(v_tX_\ell(x_t))dt.
\end{equation}
Define processes $(m_t)_{t\in[0,1]}$ in $\R$ and $(r_t)_{t\in[0,1]}$ in $\R^d$ by
\begin{align*}
dm_t
&=\sum_{\ell=1}^m\dot h_t^\ell dB^\ell_t,\q m_0=0\\
dr_t
&=\sum_{\ell=1}^mv_tX_\ell(x_t)dB^\ell_t,\q r_0=0.
\end{align*}
Write $(x_t')_{t\in[0,1]}$ for the derived process associated to the stochastic differential equation \eqref{DGE} and set
$$
y_t^{(0)}=(x_{t\wedge t_1},\dots,x_{t\wedge t_k},x_t,v_t,m_t,r_t,x_t').
$$ 
Then $(y_t^{(0)})_{t\in[0,1]}$ satisfies a stochastic differential equation of the form \eqref{ZT}.
The stopped components are obtained by multiplying the coefficients of \eqref{DGE} by the indicator function $1_{[0,t_i]}(t)$.
For $n\ge0$, recursively, set $z^{(n)}_t=(y^{(0)}_t,\dots,y^{(n)}_t)$, 
note that $(z_t^{(n)})_{t\in[0,1]}$ satisfies a stochastic differential equation of the form \eqref{ZT} with graded Lipschitz coefficients,
define $((z^{(n)})'_t)_{t\in[0,1]}$ by solving the associated derived equation, write $(z^{(n)})_t'=((y^{(0)})_t',\dots,(y^{(n)})_t')$, and set
$$
y^{(n+1)}_t=(y^{(n)})_t'.
$$
Where a process is presented in components, for example $y^{(0)}_t$, we will write the corresponding decomposition of its derived process in the obvious way.
Thus 
$$
(y^{(0)})_t'=(x_{t\wedge t_1}',\dots,x_{t\wedge t_k}',x_t',v_t',m_t',r_t',x_t''). 
$$
It is straightforward to check that this notation is consistent when a given process is a component in two different autonomous processes.
By a standard calculation, we have $x_t'=\s u_tc_t$, where $c_t$ is the Malliavin covariance matrix
\begin{equation}\label{CTM}
c_t=\sum_{\ell=1}^m\int_0^t(v_sX_\ell(x_s))\otimes(v_sX_\ell(x_s))ds.
\end{equation}
It is well known that, under the bracket condition \eqref{LAD}, the Malliavin covariance matrix $c_1$ is invertible and its inverse has moments of all orders.
The basic form of Bismut's integration by parts formula is the identity in $\R^d$
$$
\E(\nabla\phi(z_1)z_1')=\E(\phi(z_1)r_1)
$$
valid for $C_b^1$ functions $\phi$ on $\R^N$ and for $(z_t)_{t\in[0,1]}$ and $(z_t')_{t\in[0,1]}$ satisfying \eqref{ZT} and \eqref{ZPT} respectively.
Set 
$$
\tilde x_1=(x_1-\g_1)/\s,\q \tilde x_1'=x_1'/\s=u_1c_1,\q \tilde x_1''=x_1''/\s.
$$
Define random variables $y$ in $(\R^d)^*\otimes(\R^d)^*$ and $R$ in $\R$ by
$$
y=(\tilde x'_1)^{-1},\q R=(F(x_1)+\tfrac12d(x,y)^2)/\s^2+m_1/\s
$$
and define $y'$ in $(\R^d)^*\otimes(\R^d)^*\otimes\R^d$ and $R'$ in $\R^d$ by
$$
y'=-y\tilde x_1''y,\q R'=\nabla F(x_1)x_1'/\s^2+m_1'/\s.
$$
Fix $n\ge0$ and apply the integration by parts formula with $(z_t)_{t\in[0,1]}$ replaced by $(z^{(n)}_t)_{t\in[0,1]}$ 
and with $\phi(z_1)$ replaced by $f(\tilde x_1)y\phi(y,z_1^{(n)})e^{-R}$ to obtain
\begin{align*}
&\E(y\otimes(\nabla f(\tilde x_1)\tilde x_1')\phi(y,z_1^{(n)})e^{-R})
+\E(f(\tilde x_1)y'\phi(y,z_1^{(n)})e^{-R})
-\E(f(\tilde x_1)\phi(y,z_1^{(n)})(y\otimes R')e^{-R})\\
&\q+\E(f(\tilde x_1)y\otimes(\nabla_y\phi(y,z_1^{(n)})y')e^{-R})
+\E(f(\tilde x_1)y\otimes(\nabla_z\phi(y,z_1^{(n)})(z^{(n)})_1')e^{-R})\\
&\q\q=\E(f(\tilde x_1)\phi(y,z_1^{(n)})(y\otimes r_1)e^{-R}).
\end{align*}
We assume here that $f$ is $C_b^1$ and that $\phi$ is polynomial in $y$, with coefficients $C^1$ in $z_1^{(n)}$ and of polynomial growth, along with their derivatives.
Then all factors in the integrands of the preceding formula have moments of all orders.
A straightforward approximation by $C_b^1$ functions then establishes the formula.
Define a linear map $\t_i:(\R^d)^*\otimes(\R^d)^*\otimes\R^d\to\R$ by 
$$
\t_i(e_j^*\otimes e^*_{i'}\otimes e_{j'})=\d_{ii'}\d_{jj'}.
$$
Then
$$
\t_i(y\otimes(\nabla f(\tilde x_1)\tilde x_1'))=\nabla_if(\tilde x_1)
$$
so 
\begin{equation}\label{IBP}
\E(\nabla_if(\tilde x_1)\phi(y,z_1^{(n)})e^{-R})=\E(f(\tilde x_1)\nabla_i^*\phi(y,z_1^{(n+1)})e^{-R})
\end{equation}
where
\begin{align*}
\nabla_i^*\phi(y,z_1^{(n+1)})
&=\t_i(y\otimes r_1+y\tilde x''_1y+y\otimes R')\phi(y,z_1^{(n)})\\
&+\t_i(y\otimes(\nabla_y\phi(y,z_1^{(n)})y\tilde x''_1y))-\t_i(y\otimes(\nabla_z\phi(y,z_1^{(n)})(z^{(n)})_1').
\end{align*}
Take 
$$
\phi_0(y,z^{(0)}_1)=G(y,z^{(0)}_1)=G(\tilde x)=g(\tilde x_{t_1},\dots,\tilde x_{t_k}).
$$
We see inductively that \eqref{IBP} is valid for $\phi_n=\nabla^*_{i_n}\dots\nabla^*_{i_1}\phi_0$ for all $n\ge0$.
So we can iterate \eqref{IBP} to obtain, for any multi-index $\a=(i_1,\dots,i_n)$
$$
\E(\nabla^\a f(\tilde x_1)G(\tilde x)e^{-R})=\E(f(\tilde x_1)(\nabla^*)^\a G(y,z^{(n)}_1)e^{-R}).
$$
We take $f(x)=e^{i\<\xi,x\>}$ to deduce that $|\xi^\a||\hat G_\ve(\xi)|\le C_\ve(\a,G)$ where 
$$
C_\ve(\a,G)=\E(|(\nabla^*)^\a G(y,z^{(n)}_1)|e^{-R}).
$$
Now $y=c_1^{-1}v_1$ and $(\nabla^*)^\a G$ is of polynomial growth in $(y,z^{(n)}_1)$.
Given the estimate \eqref{BKE}, to prove \eqref{HRR}, it will suffice to show that, for $n=d+1$ and for all $p<\infty$, we have
\begin{equation}\label{ZPE}
\sup_{\s\in(0,1]}\E(|z^{(n)}_1|^p)<\infty
\end{equation}
and
\begin{equation}\label{CPE}
\sup_{\s\in(0,1]}\E(|c_1^{-1}|^p)<\infty.
\end{equation}

We already noted the availability of moment estimates of all orders for $z^{(n)}_1$ and that these are uniform in $\s\in[0,1]$ for the components
derived from $(v_t)_{t\in[0,1]}$ and $(r_t)_{t\in[0,1]}$.
Recall that $\tilde x_t=(x_t-\g_t)/\s$ and $R=(F(x_1)+d(x,y)^2/2)/\s^2+m_1/\s$. 
We will use first and second order mean value theorems to see that there is in fact no singularity as $\s\to0$ in these processes or any processes derived from them.
Consider, for $\th\in[0,1]$, the stochastic differential equation in $\R^d$
$$
dx_t(\th)=\sum_{\ell=1}^mX_\ell(x_t(\th))dh^\ell_t+\sum_{\ell=1}^m\th\s X_\ell(x_t(\th))dB^\ell_t+\th^2\s^2\tilde X_0(x_t(\th))dt,\q x_0(\th)=x.
$$
There exists a unique family of strong solutions $(x_t(\th))_{t\in[0,1]}$ which are almost surely jointly continuous in $\th$ and $t$.
Moreover, the following derivatives then exist for all $\th$ and $t$, almost surely
$$
\bar x_t(\th)=\frac1\s\frac\pd{\pd\th}x_t(\th),\q\hat x_t(\th)=\frac1{\s^2}\left(\frac\pd{\pd\th}\right)^2x_t(\th).
$$
Moreover, the processes $(\bar x_t(\th))_{t\in[0,1]}$ and $(\hat x_t(\th))_{t\in[0,1]}$ satisfy
\begin{align*}
d\bar x_t(\th)
&=\sum_{\ell=1}^m\nabla X_\ell(x_t(\th))\bar x_t(\th)dh^\ell_t+\sum_{\ell=1}^m\th\s\nabla X_\ell(x_t(\th))\bar x_t(\th)dB^\ell_t
+\th^2\s^2\nabla\tilde X_0(x_t(\th))\bar x_t(\th)dt\\
&\q\q+\sum_{\ell=1}^mX_\ell(x_t(\th))dB^\ell_t+2\th\s\tilde X_0(x_t(\th))dt,\q \bar x_0(\th)=0\\
d\hat x_t(\th)
&=\sum_{\ell=1}^m\nabla X_\ell(x_t(\th))\hat x_t(\th)dh^\ell_t+\sum_{\ell=1}^m\th\s\nabla X_\ell(x_t(\th))\hat x_t(\th)dB^\ell_t
+\th^2\s^2\nabla\tilde X_0(x_t(\th))\hat x_t(\th)dt\\
&\q+\sum_{\ell=1}^m\nabla^2X_\ell(x_t(\th))(\bar x_t(\th),\bar x_t(\th))dh^\ell_t
+\sum_{\ell=1}^m\th\s\nabla^2X_\ell(x_t(\th))(\bar x_t(\th),\bar x_t(\th))dB^\ell_t\\
&\q\q\q\q\q\q+\th^2\s^2\nabla^2\tilde X_0(x_t(\th))(\bar x_t(\th),\bar x_t(\th))dt\\
&\q+\sum_{\ell=1}^m2\nabla X_\ell(x_t(\th))\bar x_t(\th)dB^\ell_t+4\th\s\nabla\tilde X_0(x_t(\th))\bar x_t(\th)dt\\
&\q+2\tilde X_0(x_t(\th))dt,\q\hat x_0(\th)=0.
\end{align*}
Set $z_t(\th)=(x_t(\th),\bar x_t(\th),\hat x_t(\th))$.
Note that the process $(z_t(\th))_{t\in[0,1]}$ is the solution of a stochastic differential equation with graded Lipschitz coefficients, 
and that the coefficient bounds of the graded structure are uniform in $\th\in[0,1]$ and $\s\in[0,1]$.
Hence $(z_t(\th))_{t\in[0,1]}$ and its derived process $(z_t'(\th))_{t\in[0,1]}$ have moments of all orders, and these are bounded uniformly in $\th$ and $\s$.
Now 
$$
\tilde x_t=(x_t-\g_t)/\s=\int_0^1\bar x_t(\th)d\th
$$
and
\begin{equation}\label{EFR}
R=\int_0^1(1-\th)\{F'(x_1(\th))\hat x_1(\th)+F''(x_1(\th))(\bar x_1(\th),\bar x_1(\th))\}d\th.
\end{equation}
We can take the derivative in \eqref{DDE} under the integral sign in \eqref{EFR} to express $R'$ also by such an integral over $\th$, 
with integrand expressed in terms of $z_1(\th)$ and $z_1'(\th)$.
Hence $\sup_{t\in[0,1]}|\tilde x_t|$, $R$ and $R'$ have moments of all orders which are bounded uniformly in $\s\in(0,1]$.
The same reasoning can be extended to conclude that all processes derived from $(\tilde x_t)_{t\in[0,1]}$ and $R$ 
have moments of all orders bounded uniformly in $\s\in(0,1]$.
Alternatively, it can be observed that $(x_t,\tilde x_t',v_t/\s,\tilde x_t''/\s,m_t''/\s)$ satisfies a stochastic differential equation 
with graded Lipschitz coefficients, with bounds uniform in $\s\in(0,1]$, and from this observation we can draw the same conclusion.
Hence we have shown \eqref{ZPE}.

We turn to the proof of \eqref{CPE}.
Write $(v_t(0))_{t\in[0,1]}$ for the solution to \eqref{VGE} when $\s=0$.
Recall that
$$
c_1=\sum_{\ell=1}^m\int_0^1(v_tX_\ell(x_t))\otimes(v_tX_\ell(x_t))dt
$$
and note that $c_1$ depends continuously on $(x_t,v_t)_{t\in[0,1]}$ in uniform norm.
We have assumed that the deterministic Malliavin covariance matrix $c_1(0)$ is invertible.
Hence there are constants $C<\infty$ and $\d>0$ such that $|c_1^{-1}|\le C$ on the event
$$
\O(\d)=\{|x_t-\g_t|\le\d\q\text{and}\q|v_t-v_t(0)|\le\d\q\text{for all}\q t\in[0,1]\}.
$$
By standard $L^p$ estimates for stochastic differential equations, for all $p<\infty$, there is a constant $C(\d,p)<\infty$ such that, for all $\s\in(0,1]$,
$$
\PP(\O(\d)^c)\le C(\d,p)\s^p.
$$
On the other hand, it is shown in \cite{AKS} that, for all $p<\infty$, there are constants $C(p)<\infty$ and $D\in\N$ such that, for all $\s\in(0,1]$,
$$
\|c_1^{-1}\|_p\le C(p)\s^{-D}.
$$
Now $|c_1^{-1}|\le C+|c_1^{-1}|1_{\O(\d)^c}$ so, by H\"older's inequality, for all $p<\infty$ and all $\s\in(0,1]$,
$$
\|c_1^{-1}\|_p\le C+C(2p)C(\d,2Dp)^{1/2p}.
$$

Finally, consider the case where $G(\o)=|\o_s-\o_t|^4$ for some $s,t\in[0,1]$.
Set $\tilde x_t^{(0)}=\tilde x_t$ and, recursively for $n\ge0$, set $\tilde x_t^{(n+1)}=(\tilde x_t,(\tilde x^{(n)})'_t)$.
Then, by standard estimates, for all $p\in[1,\infty)$, there is a constant $C<\infty$ such that, uniformly in $s,t\in[0,1]$ and in $\s\in(0,1]$, we have
$$
\E\left(|\tilde x_s^{(n)}-\tilde x_t^{(n)}|^{4p}\right)\le C|s-t|^{2p}.
$$
The adjoint operators $\nabla_i^*$ have an explicit form written above, from which we deduce that, for all $n$ and for $\a=(i_1,\dots,i_n)$,
there is a random variable $K_\a$, having moments of all orders bounded uniformly in $\s\in(0,1]$, such that
$$
(\nabla^*)^\a G(y,z_1^{(n)})=K_\a|\tilde x_s^{(n)}-\tilde x_t^{(n)}|^4.
$$
Hence, by H\"older's inequality, there is a constant $C_\a<\infty$ such that, uniformly in $s,t\in[0,1]$ and in $\ve\in(0,1]$, we have
$$
C_\ve(\a,G)\le C_\a|s-t|^2.
$$
This implies \eqref{GST}.
\end{proof}

\section{Localization}\label{LL}
For $x,y\in M$ and any closed set $A$ in $M$, set
$$
d(x,A)=\inf\{d(x,z):z\in A\},\q
d(x,A,y)=\inf\{d(x,z)+d(z,y):z\in A\}.
$$
Define also the heat kernel through $A$ by
$$
p(t,x,A,y)=p(t,x,y)-p_{M\sm A}(t,x,y)
$$
where $p_{M\sm A}$ is the Dirichlet heat kernel of $\cL$ in $M\sm A$.
The following heat kernel upper bound is proved in \cite[Theorem 1.1]{BN}.
It provides a suitable estimate to deduce the second form of our main result from the first.

\begin{theorem}\label{HKUB}
Let $M$ be a connected $\cinf$ manifold.
Let $\cL$ be a second order differential operator on $M$ of the form \eqref{CFO}, whose diffusivity has a sub-Riemannian structure.
Then, for all $x,y\in M$ and any closed set $A$ in $M$ with $M\sm A$ relatively compact, we have
\begin{equation*}
\limsup_{t\to0}t\log p(t,x,A,y)\le-(d(x,A)+d(y,A))^2/2.
\end{equation*}
Moreover, if $\cL$ also satisfies the sector condition \eqref{SECT} then, for any closed set $A$ in $M$,
\begin{equation*}
\limsup_{t\to0}t\log p(t,x,A,y)\le-d(x,A,y)^2/2.
\end{equation*}
Moreover, all the above upper limits hold uniformly in $x$ and $y$ in compact
subsets of $M\sm\pd A$.
\end{theorem}

\begin{proof}[Proof of Theorem \ref{MR}]
There exists an open set $U\sse M$, which contains the minimal path $\g$, and which is compactly contained in a chart $\phi:U_0\to\R^d$ of $M$.
In the case where $d(x,y)<d(x,\infty)+d(y,\infty)$ we can and do choose $U$ so that, for $A=M\sm U$, we have
$$
d(x,y)<d(x,A)+d(y,A).
$$
It will suffice, and it will lighten the notation, to consider the case where $U_0$ is a domain both in $M$ and in $\R^d$, where $\phi$ is the identity map, and where $\th$ restricts to the identity map on $U$.
We will show that there exist vector fields $\bar X_0,\bar X_1,\dots,\bar X_{m+d}$ on $\R^d$, which are bounded with bounded derivatives of all orders and such that 
\begin{itemize} 
\item[{\rm (a)}] $(\bar X_1,\dots,\bar X_{m+d})$ is a sub-Riemannian structure on $\R^d$,
\item[{\rm (b)}] we have $a=\bar a$ and $\cL=\bar\cL$ on $U$, where
$$
\bar\cL=\frac12\sum_{\ell=1}^{m+d}\bar X_\ell^2+\bar X_0
$$
and where $a$ and $\bar a$ are the diffusivities of $\cL$ and $\bar\cL$ respectively,
\item[{\rm (c)}] $\g$ is the unique $\bar a$-minimal path in $\hxy(\R^d)$.
\end{itemize}
From these properties, given that $(x,y)$ lies outside the cut locus of $\cL$ in $M$, we deduce that $(x,y)$ also lies outside the cut locus of $\bar\cL$ in $\R^d$.
Write $\bar p$ for the heat kernel of $\bar\cL$ on $\R^d$ and write $\mexyd$ for the bridge measure on $\oxy(\R^d)$ associated to the operator $\ve\bar\cL$.
Similarly, write $\tmexyd$ for the rescaled bridge measure on $\tgoxy$, given by $\tmexyd=\mexyd\circ\bar\s_\ve^{-1}$, where $\bar\s_\ve(\o)_t=(\o_t-\g_t)/\sqrt\ve$.
Then, as we showed in Section \ref{LAP}, as $\ve\to0$, we have 
$$
\bar p(\ve,x,y)=c(x,y)\ve^{-d/2}\exp\{-d(x,y)^2/(2\ve)\}(1+o(1))
$$
and $\tmexyd\to\mg$ weakly on $\tgoxy$.

Write $\mexyU$ for the diffusion bridge measure on $\oxy(U)$ associated to the restriction of the operator $\ve\cL$ to $U$.
Then, for all measurable sets $S$ in $\oxy(M)$, we have
$$
p(\ve,x,y)\mexy(S)=p_U(\ve,x,y)\mexyU(S\cap\oxy(U))+p(\ve,x,y)\mexy(S\sm\oxy(U))
$$
and
$$
\bar p(\ve,x,y)\mexyd(S\cap\oxy(U))=p_U(\ve,x,y)\mexyU(S\cap\oxy(U)).
$$
For any bounded measurable set $B$ in $\tgoxy$ and for $\ve>0$ sufficiently small, we have $\s_\ve^{-1}(B)=\bar\s_\ve^{-1}(B)\sse\oxy(U)$, so
$$
p(\ve,x,y)\tmexy(B)=\bar p(\ve,x,y)\tmexyd(B).
$$
Now $\mexyd(\oxy(U))\to1$ so, taking $S=\oxy(M)$, we find
$$
p_U(\ve,x,y)=c(x,y)\ve^{-d/2}\exp\{-d(x,y)^2/(2\ve)\}(1+o(1)).
$$
On the other hand, by Theorem \ref{HKUB}, we have either
$$
\limsup_{\ve\to0}p(t,x,A,y)\le-(d(x,A)+d(y,A))^2/2<-d(x,y)^2/2
$$
or $\g$ is strongly minimal and $\cL$ satisfies \eqref{SECT}, so
$$
\limsup_{\ve\to0}p(t,x,A,y)\le-d(x,A,y)^2/2<-d(x,y)^2/2.
$$
In any case, we obtain
$$
p(\ve,x,y)=p_U(\ve,x,y)+p(\ve,x,A,y)=c(x,y)\ve^{-d/2}\exp\{-d(x,y)^2/(2\ve)\}(1+o(1)).
$$
Moreover, for all $B\in\tgoxy$, we then have $\tmexy(B)=\tmexyd(B)(1+o(1))$, so the claimed weak limit for $\tmexy$ follows from that for $\tmexyd$.

It remains to show the existence of vector fields $\bar X_0, \bar X_1,\dots,\bar X_{m+d}$ with the claimed properties.
Fix an open set $U_1$ such that $U$ is compactly contained in $U_1$ and $U_1$ is compactly contained in $U_0$.
There exists a $\cinf$ function $\chi$ on $\R^d$ such that $1_U\le\chi\le1$ and which is uniformly positive on $U_1$ and such that $\{\chi>0\}$ is compactly contained in $U_0$.
Write $X_0$ for the vector field on $M$ given by \eqref{HSS}.
For $\ell=0,1,\dots,m$, set $\bar X_\ell(z)=\chi(z)X_\ell(z)$ for $z\in U_0$ and set $\bar X_\ell(z)=0$ for $z\in\R^d\sm U_0$.
Then $\bar X_\ell$ is a bounded vector field on $\R^d$ with bounded derivatives of all orders.
Define $\tilde a=\sum_{\ell=1}^m\bar X_\ell\otimes\bar X_\ell$ and write $\tilde I$ for the associated energy function.
Then $\tilde I(\o)\ge I(\o)$ for all $\o\in\hxy(U_0)$ with equality if $\o$ is contained in $U$.
Moreover, by choosing $\chi$ sufficiently small near $\pd U_1$, we can and do ensure that, for all $z\in\pd U_1$ and all $\o\in H^{x,z}(U_0)$, 
we have $\tilde I(\o)>I(\g)$.
Choose now another $\cinf$ function $\bar\chi$ on $\R^d$ such that $1_{U_1}\le1-\bar\chi\le1_{U_0}$ and such that $\chi+\bar\chi$ is everywhere positive.
For $i=1,\dots,d$, set $\bar X_{m+i}(z)=\bar\chi(z)e_i$, where $e_1,\dots,e_d$ is the standard basis in $\R^d$.
Then the brackets of $\bar X_1,\dots,\bar X_m$ span $\R^d$ on $\{\chi>0\}$, while $\bar X_{m+1},\dots,\bar X_{m+d}$ themselves span $\R^d$ on $\{\bar\chi>0\}$.
Hence (a) holds.
Also, (b) holds because $\bar X_\ell=X_\ell$ on $U$ for $\ell=0,1,\dots,m$ and $\bar X_\ell=0$ on $U$ for $\ell=m+1,\dots,m+d$.
Now, if $\o\in\hxy(\R^d)$ is contained in $U_1$, then $\bar I(\o)=\tilde I(\o)\ge I(\o)$, so $\bar I(\o)\ge I(\g)$ with equality only if $\o=\g$.
On the other hand, if $\o$ is not contained in $U_1$, set $\t=\inf\{t\in[0,1]:\o_\t\in\pd U_1\}$ and set $\tilde\o_t=\o_{\t t}$.
Then $\tilde\o\in H^{x,z}(U_0)$, where $z=\o_\t\in\pd U_1$.
So $\bar I(\o)\ge \tilde I(\tilde\o)>I(\g)$.
Hence (c) holds.
\end{proof}

\def\j{
\section{Heat kernel upper bounds}\label{NFB}
Let $M$ be a connected $\cinf$ manifold.
Consider a second order differential operator $\cL$ on $M$ of the form \eqref{CFO}
$$
\cL f=\tfrac12\dvv(a\nabla f)+a(\b,\nabla f)
$$
where the diffusivity $a$ has a sub-Riemannian structure $(X_1,\dots,X_m)$,
where the divergence is understood with respect to a positive $\cinf$ locally invariant measure $\nu$ and where $\b$ is a $\cinf$ $1$-form satisfying the sector condition \eqref{SECT}.
Recall the definition \eqref{DIST} of the sub-Riemannian distance.
For a closed subset $A$ of $M$ and $x,y\in M\sm A$, consider the distance from $x$ to $y$ through $A$, given by
$$
d(x,A,y)=\inf\{\sqrt{I(\o)}:\o\in\hxy,\,\o_t\in A\text{ for some $t\in[0,1]$}\}.
$$
The heat kernel for diffusion from $x$ to $y$ through $A$ is defined by
$$
p(t,x,K,y)=p(t,x,y)-p_{M\sm K}(t,x,y)
$$
where $p_{M\sm K}$ is the Dirichlet heat kernel of $\cL$ in $M\sm K$.
We will prove in this section the following upper bound.\footnote{%
In the case where $a$ is everywhere positive-definite, there is an alternative argument of Hsu \cite{MR1089046} for heat kernel upper bounds in incomplete manifolds, 
relying  on estimates of Azencott 
\cite{MR634964}, 
which does not require the condition \eqref{SECT}. 
Hsu instead requires that $d(x,y)\le d(x,\infty)+d(\infty,y)$. 
The same paper \cite{MR1089046} also contains an example where heat flow does not localize near the shortest path.}
\begin{theorem}\label{HTK}
For all $x,y\in M$ and all closed subsets $K$ of $M$, we have
$$
\limsup_{t\to0}t\log p(t,x,K,y)\le-d(x,K,y)^2/2.
$$
\end{theorem}

Before giving the proof, we show how this estimate allows us to prove Theorems \ref{VAR}, \ref{BALOC} and \ref{MR}.

\begin{proof}[Proof of Theorem \ref{VAR}]
The upper bound in \eqref{LOGP} is the simple case $K=M$ of Theorem \ref{HTK}, 
while the lower bound follows by a standard argument from L\'eandre's lower bound in $\R^d$.

We turn to the proof of \eqref{LDEV}. 
Consider the operator 
$$
\tilde\cL=\cL+\tfrac12(\pd/\pd\t)^2
$$ 
on $\tilde M=M\times\R$, where $\t$ denotes the coordinate in $\R$.
Then the diffusivity of $\tilde\cL$ also has a sub-Riemannian structure.
Set $\tilde x=(x,0)$ and $\tilde y=(y,1)$, and define
$$
\tilde K=\tilde M\sm\tilde U,\q \tilde U=\{(\g_t,\s_t):(\g,\s)\in\tilde\G(\d),\,t\in[0,1]\}
$$
where
$$
\tilde\G(\d)=\left\{(\g,\s)\in\hxy\times H^{0,1}(\R):I(\g)+\int_0^1\dot\s_t^2dt<d(x,y)^2+1+\d\right\}.
$$
Then $\tilde K$ is closed in $\tilde M$.
Write $\b^{0,1}_\ve$ for the law on $\O^{0,1}(\R)$ of a Brownian bridge from $0$ to $1$ of speed $\ve$.
Then, with obvious notation,
$$
\tilde p(t,\tilde x,\tilde y)=p(t,x,y)\frac1{\sqrt{2\pi}}e^{-1/(2t)},\q\tilde\mu^{\tilde x,\tilde y}_\ve(d\o,d\t)=\mexy(d\o)\b_\ve^{0,1}(d\t).
$$
By Theorem \ref{HTK}, we have
$$
\limsup_{t\to0}t\log\tilde p(t,\tilde x,\tilde K,\tilde y)\le-\tilde d(\tilde x,\tilde K,\tilde y)^2/2=-(d(x,y)^2+1+\d)/2
$$
so
\begin{align}\notag
&\limsup_{\ve\to0}\ve\log\tilde\mu_\ve^{\tilde x,\tilde y}(\{(\o,\t):(\o_t,\t_t)\in\tilde K\text{ for some }t\in[0,1]\})\\
\label{LD1}
&\q\q\q\q\le\limsup_{\ve\to0}\ve\log\tilde p(\ve,\tilde x,\tilde K,\tilde y)-\liminf_{\ve\to0}\ve\log\tilde p(\ve,\tilde x,\tilde y)\le-\d/2
\end{align}
where we have used the lower bound from \eqref{LOGP}.
By standard estimates, we also have
\begin{equation}\label{LD2}
\lim_{\ve\to0}\ve\log\b_\ve^{0,1}(\{\t:|\t_t-t|\ge\sqrt\d/2\text{ for some }t\in[0,1]\})=-\d/2.
\end{equation}
Suppose then that $\o\in\oxy$ and $\t\in\O^{0,1}(\R)$ satisfy $(\o_t,\t_t)\in\tilde U$ and $|\t_t-t|<\sqrt\d/2$ for all $t\in[0,1]$.
Then, for each $t\in[0,1]$, there exist $s\in[0,1]$ and $\g\in\hxy$ and $\s\in H^{0,1}(\R)$ such that 
$$
\o_t=\g_s,\q \t_t=\s_s,\q I(\g)<d(x,y)^2+\d,\q \int_0^1\dot\s_r^2dr<1+\d.
$$
Then $|\s_s-s|\le\sqrt\d/2$ so $|t-s|\le\sqrt\d$ and so 
$$
d(\o_t,\G_t(\d))^2\le d(\o_t,\g_t)^2=d(\g_s,\g_t)^2\le|t-s|I(\g)\le\d^{1/2}(d(x,y)^2+\d).
$$
The estimates \eqref{LD1} and \eqref{LD2} thus imply \eqref{LDEV}.
\end{proof}  

\begin{proof}[Proof of Theorems \ref{BALOC} and \ref{MR}]
The hypothesis that $\g\in\hxy$ is strongly minimal implies that, for any domain $U$ containing $\g$, we have $d(x,M\sm U,y)>d(x,y)$.
We can choose a chart $U_0$ in $M$ and a domain $U$ containing $\g$ such that $U$ is compactly contained in $U_0$.
We suppose that $\cL$ has the form \eqref{CFO} and satisfies the sector condition \eqref{SECT}.
Then Theorem \ref{HTK} applies to show that
$$
\limsup_{t\to0}t\log p(t,x,M\sm U,y)\le-d(x,M\sm U,y)^2/2<-d(x,y)^2/2
$$
and Theorem \ref{NFBT} then gives the desired conclusions.
\end{proof}

In proving Theorem \ref{HTK}, we will use the following estimate of Nagel, Stein \& Wainger \cite{MR793239} for the sub-Riemannian distance.
There is a covering of $M$ by relatively compact charts $\phi:U\to\R^d$ such that, for some constants $\a(U)>0$ and $C(U)<\infty$, for all $x,y\in U$, we have
$$
C^{-1}|\phi(x)-\phi(y)|\le d(x,y)\le C|\phi(x)-\phi(y)|^\a.
$$
If follows that, for any positive $\cinf$ measure $\nu$ on $M$, we can choose $C(U)$ so that moreover, 
for all $x\in U$ and all $r\in(0,\infty)$ such that $B(x,r)=\{y\in M:d(x,y)<r\}\sse U$, we have
\begin{equation}\label{VE}
C^{-1}r^{d/\a}\le\nu(B(x,r))\le Cr^d
\end{equation}
Moreover, Nagel, Stein \& Wainger show also that $C(U)$ may be chosen so that the following local volume-doubling inequality holds: 
for all $x\in U$ and all $r\in(0,\infty)$ such that $B(x,2r)\sse U$, we have
\begin{equation}\label{DP}
\nu(B(x,2r))\le C\nu(B(x,r)).
\end{equation}
We will also need the local Poincar\'e inequality proved by Jerison \cite{MR850547}, which also relies on the fact that $a$ has a sub-Riemannian structure.
There is a covering of $M$ by open sets $U$ such that, for some constant $C(U)<\infty$, 
for all $x\in U$ and all $r\in(0,\infty)$ such that $B(x,2r)\sse U$, for all $f\in\cinf_c(M)$, we have
\begin{equation}\label{PI}
\int_{B(x,r)}|f-\<f\>_{B(x,r)}|^2d\nu\le Cr^2\int_{B(x,2r)}a(\nabla f,\nabla f)d\nu
\end{equation}
where $\<f\>_B=\int_Bfd\nu/\nu(B)$ is the average value of $f$ on $B$.
The sector condition \eqref{SECT} gives us uniform parabolicity in the sense of Sturm, so \eqref{DP} and \eqref{PI} allow us to apply \cite[Theorem 2.1]{MR1355744}
to obtain the following parabolic mean-value estimate.
For simplicity, we will state a simple version adequate for our needs.
Let $x\in M$ and let $r_0\in(0,\infty)$.
Let $D$ be a domain in $M$ which compactly contains the ball $B(x,r_0)$.
Suppose that $u$ is a non-negative weak solution of the heat equation $(\pd/\pd t)u_t=\cL u_t$ on $(0,\infty)\times D$.
Then there is constant $C(D,r_0)<\infty$ such that, for all $t\in(0,\infty)$ and all $r\in(0,r_0]$ with $r^2\le t/2$, we have
\begin{equation}\label{MVE}
u_t(x)^2\le C\fint_{t-r^2}^t\fint_{B(x,r)}u_s^2d\nu ds.
\end{equation}
See Lierl and Saloff-Coste \cite[Theorem 2.6]{1205.6493} for an alternative proof, which clarifies some points in the non-symmetric case.

The sub-Riemannian distance has a dual formulation, proved in Jerison \& Sanchez-Calle \cite{MR922334}, 
which we adapt to the case of a general sub-Riemannian manifold, without completeness, and to the distance through $K$.
\begin{proposition}\label{DUAL}
For all $x,y\in M$ and any closed subset $K$ of $M$, we have
$$
d(x,K,y)=\sup\{w^+(y)-w^-(x):w^-,w^+\in\cF\text{ with $w^+=w^-$ on $K$}\}
$$
where $\cF$ is the set of all locally Lipschitz functions $w$ on $M$ such that $a(\nabla w,\nabla w)\le1$ almost everywhere.
\end{proposition}
\begin{proof}
Denote the right hand side by $\d(x,K,y)$ for now.
Suppose that $\o\in\hxy$ has driving path $\xi$ and that $\o_t\in K$.
Let $w^-,w^+\in\cF$, with $w^+=w^-$ on $K$.
It will suffice to consider the case where $\o$ is simple, when there exists a relatively compact chart $U$ for $M$ containing $\o$.
Then, given $\ve>0$, since $a$ is continuous, we can find $\cinf$ functions $f^-,f^+$ on $U$ such that $|f^\pm(z)-w^\pm(z)|\le\ve$ and $a(\nabla f^\pm,\nabla f^\pm)(z)\le1+\ve$ for all $z\in U$.
Then
$$
w^+(y)-w^-(x)=w^+(y)-w^+(\o_t)+w^-(\o_t)-w^-(x)\le f^+(y)-f^+(\o_t)+f^-(\o_t)-f^-(x)+4\ve
$$
and
\begin{align*}
&f^+(y)-f^+(\o_t)+f^-(\o_t)-f^-(x)\\
&\q\q=\int_0^t\<\nabla f^-(\o_s),\dot\o_s\>ds+\int_t^1\<\nabla f^+(\o_s),\dot\o_s\>ds\\
&\q\q=\int_0^t\<\nabla f^-(\o_s),a(\o_s)\xi_s\>ds+\int_t^1\<\nabla f^+(\o_s),a(\o_s)\xi_s\>ds\\
&\q\q\le\left(\int_0^ta(\nabla f^-,\nabla f^-)(\o_s)ds+ \int_t^1a(\nabla f^+,\nabla f^+)(\o_s)ds\right)^{1/2}\left(\int_0^1a(\xi_s,\xi_s)ds\right)^{1/2}\\
&\q\q\le\sqrt{(1+\ve)I(\o)}.
\end{align*}
Hence $w^+(y)-w^-(x)\le\sqrt{I(\o)}$.
On taking the supremum over $w^\pm$ and the infimum over $\o$, we deduce that $\d(x,K,y)\le d(x,K,y)$.

Now we prove the reverse inequality.
Choose a $\cinf$ quadratic form $\a$ on $T^*M$ which is everywhere positive-definite.
Consider the energy function $I_\a$ and the distance functions $d_\a$ and $\d_\a$ obtained by replacing $a$ by $a+\a$ in the definitions of $d$ and $\d$.
Set $w^+(z)=d_\a(x,K,z)$ and $w^-(z)=d_\a(x,z)$.
Note that $w^+=w^-$ on $K$.
Since $a+\a$ is positive-definite, the functions $w^-$ and $w^+$ are locally Lipschitz, and their weak gradients $\nabla w^\pm$ satisfy $(a+\a)(\nabla w^\pm,\nabla w^\pm)\le1$ almost everywhere.
Hence 
$$
d_\a(x,K,y)=w^+(y)-w^-(x)\le\d_\a(x,K,y)\le\d(x,K,y)
$$
and we can complete the proof by finding, for each $\ve\in(0,1]$, a choice of $\a$ so that $d(x,K,y)\le(1+\ve)(d_\a(x,K,y)+\ve)+\ve$.

Choose $\phi^\a\in\oxy$, passing through $K$ and such that $\sqrt{I_\a(\phi^\a)}\le d_\a(x,K,y)+\ve$.
There exist $p\in\N$ and $\cinf$ vector fields $Y_1,\dots,Y_p$ on $M$ such that $\a=\sum_{i=1}^pY_i\otimes Y_i$.
Then there exist $h\in\hrm$ and $k\in\hrp$ such that $\phi^\a$ has weak derivative\footnote{%
See the discussion preceding Proposition \ref{2.1}.}
$$
\dot\phi_t^\a=\sum_{\ell=1}^mX_\ell(\phi^\a_t)\dot h^\ell_t+\sum_{i=1}^pY_i(\phi_t^\a)\dot k_t^i.
$$
By reparametrizing $\phi^\a$ if necessary, we may assume that $|\dot h_t|^2+|\dot k_t|^2=I_\a(\phi^\a)$ for almost all $t$.
It will be convenient to assume that $d(x,\infty)\ge1/2$, which we may do without loss of generality by a scaling argument.
Define $S_{-1}=S_0=\es$ and consider for $n\ge1$ the compact set
$$
S_n=\{z\in M:d(z,x)\le 2^n\text{ and }d(z,\infty)\ge2^{-n}\}.
$$
Define, recursively, sequences of times $t_0,t_1,\dots,t_J$, points $x_1,\dots,x_J$ and $y_1,\dots,y_J$, and positive integers $n_1,\dots,n_J$ as follows.  
Set $t_0=0$ and $n_1=1$.
For $j\ge1$, set $x_j=\phi^\a_{t_{j-1}}$, define $(\phi^j_t)_{t\ge t_{j-1}}$ by the differential equation
$$
\dot\phi_t^j=\sum_{\ell=1}^mX_\ell(\phi_t^j)\dot h^\ell_t,\q\phi_{t_{j-1}}^j=x_j
$$
and set $y_j=\phi_{t_j}^j$, where
$$
t_j=\inf\{t\ge t_{j-1}:\phi^j_t\in\pd S_{n_j-1}\cup\pd S_{n_j+1}\}\wedge1.
$$
If $t_j<1$, then set 
$$
n_{j+1}=
\begin{cases}
n_j-1,&\text{if }y_j\in\pd S_{n_j-1},\\
n_j+1,&\text{if }y_j\in\pd S_{n_j+1}.
\end{cases}
$$
Thus $y_j\in\pd S_{n_{j+1}}$.
If $t_j=1$, then set $J=j$ and stop.
We will see below that the recursion does stop.
Note that, for all $j$ and all $s,t\in[t_{j-1},t_j]$ with $s\le t$, we have
$$
d(\phi_s^j,\phi_t^j)\le(t-s)\|\dot h\|_\infty.
$$
Note also that $x_1=\phi^\a_0=x\in S_{n_1}\sse S_{n_1+1}\sm S_{n_1-1}$ and $\phi_{t_{j-1}}^\a=\phi_{t_{j-1}}^j=x_j$ for all $j$.
Suppose, inductively for $j\ge1$, that $x_j\in S_{n_j+1}\sm S_{n_j-1}$.
Then $\phi^j_t$ remains in this set for $t\in[t_{j-1},t_j)$ and, by choosing $\a$ sufficiently small on $S_{n_j+2}\sm S_{n_j-2}$, we can ensure that 
$$
d(\phi_t^j,\phi_t^\a)\le2^{-2n_j-2}\ve\le2^{-n_{j+1}-2},\q t\in[t_{j-1},t_j]. 
$$
Then, in particular, we have $d(y_j,x_{j+1})\le 2^{-n_{j+1}-2}$.
Now $y_j\in\pd S_{n_{j+1}}$ so this implies that $x_{j+1}\in S_{n_{j+1}+1}\sm S_{n_{j+1}-1}$ and the induction proceeds.
Also, since $y_{j+1}\in\pd S_{n_{j+2}}$, we have $d(y_j,y_{j+1})\ge2^{-n_{j+1}-1}$ and so $d(x_{j+1},y_{j+1})\ge2^{-n_{j+1}-2}$.
Certainly $d(x_1,y_1)\ge2^{-n_1-2}$ so, for all $j\ge1$, we have $d(x_j,y_j)\ge2^{-n_j-2}$.
Hence
$$
(t_j-t_{j-1})\|\dot h\|_\infty\ge2^{-n_j-2}
$$
and so, for all $n\ge1$, we have $n_j=n$ at most $2^{n+2}\|\dot h\|_\infty$ times.
In particular, the recursion must stop, or $n_j\to\infty$ as $j\to\infty$ forcing $\phi_{t_j}^\a$ to leave all compact sets, which is impossible.
Hence
$$
\sum_{j=1}^{J-1}d(y_j,x_{j+1})+d(y_J,y)\le\sum_{n=1}^\infty2^{n+2}\|\dot h\|_\infty.2^{-2n-2}\ve=\ve\|\dot h\|_\infty.
$$
Also, there exist $j$ and $t$ so that $t\in[t_{j-1},t_j]$ and $\phi_t^\a\in K$, and then
$$
d(x_j,K,y_j)\le d(x_j,\phi_t^j)+2d(\phi_t^j,\phi_t^\a)+d(\phi_t^j,y_j)\le(t_j-t_{j-1})\|\dot h\|_\infty+\ve
$$
while, for $k\not=j$, we have the estimate
$$
d(x_k,y_k)\le(t_k-t_{k-1})\|\dot h\|_\infty.
$$
We finally combine these estimates, using the triangle inequality, to obtain
$$
d(x,K,y)\le(1+\ve)\|\dot h\|_\infty+\ve\le(1+\ve)(d_\a(x,K,y)+\ve)+\ve.
$$
\end{proof}

\begin{proof}[Proof of Theorem \ref{HTK}]
We will show that the argument used in \cite[Theorem 1.2]{MR1484769}, for the case where $a$ is positive-definite and $\b=0$, generalizes to the present context\footnote{%
The idea is to combine a standard argument for heat kernel upper bounds with a reflection trick. 
In terms of Markov processes, we give a random sign to each excursion of the diffusion process into $D$, viewing it as taking values in $D^-$ or $D^+$.
Then a generalization of the classical reflection principle for Brownian motion allows to express the density for paths from $x$ to $y$ via $K$ in terms of this enhanced process.
In fact the heat kernel $\tilde p$ for this process may be written in terms of $p$ and $p_D$, and we find it technically simpler to define $\tilde p$ in those terms, rather than set up the enhanced process.}
.
Consider the set $\tilde M=M^-\cup M^+$, where $M^\pm=K\cup D^\pm$ and $D^-,D^+$ are disjoint copies of $D=M\sm K$. 
Write $\pi$ for the obvious projection $\tilde M\to M$.
For functions $f$ defined on $M$, we will write $f$ also for the function $f\circ\pi$ on $\tilde M$.
Thus we will sometimes consider $a$ as a quadratic form on $T^*D^\pm$ and $\b$ as a $1$-form on $D^\pm$.
Define a measure $\tilde\nu$ on $\tilde M$ by
$$
\tilde\nu(A)=\nu(A\cap K)+\tfrac12\nu(\pi(A\cap D^-))+\tfrac12\nu(\pi(A\cap D^+)).
$$
Note that $\nu=\tilde\nu\circ\pi^{-1}$.
Now define
$$
\tilde p(t,x,y)=
\begin{cases}
p(t,x,y)+p_D(t,x,y),&\text{ if }x,y\in D^\pm,\\
p(t,x,y)-p_D(t,x,y),&\text{ if }x\in D^\pm\text{ and }y\in D^\mp,\\
p(t,x,y),&\text{ if }x\in K\text{ or }y\in K.
\end{cases}
$$
Given functions $f^-,f^+\in\cinf_c(M)$ with $f^-=f^+$ on $K$, write $f$ for the function on $\tilde M$ such that $f=f^\pm\circ\pi$ on $M^\pm$,
and set $\bar f=(f^-+f^+)/2$ and $f^D=(f^+-f^-)/2$.
Let $\phi^\pm\in\cinf_c(M)$ with $\phi^-=\phi^+$ on $K$ and define $\phi$ on $\tilde M$ and $\bar\phi$ and $\phi^D$ on $M$ similarly.
Define for $t\in(0,\infty)$ functions $u_t$ on $\tilde M$, $\bar u_t$ on $M$ and $u^D_t$ on $D$ by
$$
u_t(x)=\int_{\tilde M}\tilde p(t,x,y)f(y)\tilde\nu(dy)
$$
and
$$
\bar u_t(x)=\int_Mp(t,x,y)\bar f(y)\nu(dy),\q u^D_t(x)=\int_Mp_D(t,x,y)f^D(y)\nu(dy).
$$
Then $\bar u_t$ and $u^D_t$ solve the heat equation with Dirichlet boundary conditions in $M$ and $D$ respectively.
It is straightforward to check that $u_t=u^\pm_t\circ\pi$ on $M^\pm$, where $u_t^\pm=\bar u_t\pm u_t^D$ and we extend $u_t^D$ by $0$ on $K$.
Hence
$$
\int_{\tilde M}\phi u_td\tilde\nu=\int_M\bar\phi\bar u_td\nu+\int_D\phi^Du_t^Dd\nu
$$
and so
\begin{align}
\notag
&\frac{d}{dt}\int_{\tilde M}\phi u_td\tilde\nu
=\frac{d}{dt}\int_M\bar\phi\bar u_td\nu+\frac{d}{dt}\int_D\phi^D u_t^Dd\nu\\
\notag
&=-\frac12\int_Ma(\nabla\bar\phi,\nabla\bar u_t)d\nu+\int_Ma(\bar\phi\b,\nabla\bar u_t)d\nu
-\frac12\int_Da(\nabla\phi^D,\nabla u_t^D)d\nu+\int_Da(\phi^D\b,\nabla u_t^D)d\nu\\
\label{WHE}
&=-\frac12\int_{\tilde M}a(\nabla\phi,\nabla u_t)d\tilde\nu+\int_{\tilde M}a(\phi\b,\nabla u_t)d\tilde\nu.
\end{align}

Fix $x,y\in D$ and $r_0\in(0,\infty)$ so that the balls $B(x,r_0)$ and $B(y,r_0)$ are compactly contained in $D$.
Write $x^-$ and $y^+$ for the unique points in $D^-$ and $D^+$ respectively such that $\pi(x^-)=x$ and $\pi(y^+)=y$.
Fix $r\in(0,r_0]$ and set
$$
B^-=\{z\in D^-:\pi(z)\in B(x,r)\},\q B^+=\{z\in D^+:\pi(z)\in B(y,r)\}.
$$
Consider the set $\cF_r$ of pairs of bounded locally Lipschitz functions $(w^-,w^+)$ on $M$ such that $w^-$ is constant on $B(x,r)$, $w^+$ is constant on $B(y,r)$, 
$w^-=w^+$ on $K$ and $a(\nabla w^\pm,\nabla w^\pm)\le1$ almost everywhere.
Set 
$$
d_r(x,K,y)=\sup\{w^+(y)-w^-(x):(w^-,w^+)\in\cF_r\}.
$$
Fix $(w^-,w^+)\in\cF_r$ and define a function $w$ on $\tilde M$ by setting $w=w^\pm\circ\pi$ on $M^\pm$.
Take $f^-=0$ and choose $f^+\ge0$ supported on $B(y,r)$ and such that $\int_M(f^+)^2d\nu=2$.
Then $\int_{\tilde M}f^2d\tilde\nu=1$.
Fix $\d\in[0,\infty)$ and $\th\in\R$ and set $\psi=\th w$.
We deduce from \eqref{WHE} by a standard argument that
\begin{align*}
\frac d{dt}\int_{\tilde M}(e^\psi u_t)^2d\tilde\nu
&=-\int_{\tilde M}a(\nabla(e^{2\psi}u_t),\nabla u_t)d\tilde\nu+2\int_{\tilde M}a(\b e^{2\psi}u_t,\nabla u_t)d\tilde\nu\\
&=-\int_{\tilde M}a(\nabla u_t,\nabla u_t)e^{2\psi}d\tilde\nu+2\int_{\tilde M}a((\b-\nabla\psi)u_t,\nabla u_t)e^{2\psi}d\tilde\nu\\
&\le\int_{\tilde M}a(\b-\nabla\psi,\b-\nabla\psi)(e^\psi u_t)^2d\tilde\nu\le \l\int_{\tilde M}(e^\psi u_t)^2d\tilde\nu
\end{align*}
where $\l=\|a(\b-\nabla\psi,\b-\nabla\psi)\|_\infty\le(1+\d)\th^2+(1+1/\d)\|a(\b,\b)\|_\infty$. 
So, by Gronwall, we have
$$
\int_{\tilde M}(e^\psi u_t)^2d\tilde\nu\le e^{\l t}\int_{\tilde M}(e^\psi f)^2d\tilde\nu=e^{\l t+2\psi(y^+)}
$$
and so
$$
\int_{B^-}u_t^2d\tilde\nu\le e^{2\th(w^+(y)-w^-(x))+\l t}.
$$
Moreover, since $u^-_t\ge0$ and $(\pd/\pd t)u_t^-=\cL u_t^-$ on $(0,\infty)\times D$, by the parabolic mean-value estimate, there is a constant $C(D,r_0)<\infty$ such that, 
for all $t\in(0,\infty)$ and all $r\in(0,r_0]$ with $r^2\le t/2$, we have
$$
u_t(x^-)^2
\le C\fint_{t-r^2}^t\fint_{B^-}u_s^2d\tilde\nu ds
\le C\nu(B(x,r))^{-1}e^{2\th(w^+(y)-w^-(x))+\l t}.
$$
Set $v_t(z)=p(t,x,K,z)$, then $(\pd/\pd t)v_t=\hat\cL v_t$ on $(0,\infty)\times D$, where $\hat\cL f=\frac12\dvv(a\nabla f)-a(\b,\nabla f)$.
So, by the parabolic mean-value estimate again, we can choose $C(D,r_0)$ so that
$$
p(t,x,K,y)^2\le 
C\fint_{t-r^2}^t\fint_{B(y,r)}p(s,x,K,z)^2\nu(dz)ds
=C\fint_{t-r^2}^t\fint_{B^+}\tilde p(s,x^-,z)^2\tilde\nu(dz)ds.
$$
Assume now that $r^2\le t/4$.
For each $s\in[t-r^2,t]$, we may take $f^+=cp(s,x,K,.)1_{B(y,r)}$, where $c$ is chosen so that $\int_{\tilde M}f^2d\tilde\nu=1$.
For this choice of $f^+$, we have
$$
u_s(x^-)^2=\int_{B^+}\tilde p(s,x^-,z)^2\tilde\nu(dz).
$$
Hence
\begin{align*}
p(t,x,K,y)^2
&\le C\nu(B(y,r))^{-1}\fint_{t-r^2}^tu_s(x^-)^2ds\\
&\le C^2\nu(B(x,r))^{-1}\nu(B(y,r))^{-1}e^{2\th(w^+(y)-w^-(x))+\l t}
\end{align*}
from which we obtain, on optimizing over $\th$ and $w$,
$$
p(t,x,K,y)\le C\nu(B(x,r))^{-1/2}\nu(B(y,r))^{-1/2}
\exp\left\{-\frac{d_r(x,K,y)^2}{2(1+\d)t}+\left(1+\frac 1\d\right)\frac{\|a(\b,\b)\|_\infty t}2\right\}.
$$
For $(w^-,w^+)\in\cF_0$ and $r\in(0,r_0]$, set $\bar w_x=\sup_{z\in B(x,r)}w^-(z)$ and $\bar w_y=\inf_{z\in B(y,r)}w^+(z)$ and suppose that $\bar w_x\le\bar w_y$.
We can define $(w_r^-,w_r^+)\in\cF_r$ by setting $w_r^-(z)=\bar w_x\vee w(z)\wedge\bar w_y$ and $w_r^+(z)=\bar w_x\vee w(z)\wedge\bar w_y$.
Note that $|w^-(z)-w^-(x)|\le r$ for all $z\in B(x,r)$ and $|w^+(z)-w^+(y)|\le r$ for all $z\in B(y,r)$, so $w^+(y)-w^-(x)\le w_r^+(y)-w_r^-(x)+2r$.
On optimizing over $w$, we see that
$$
d(x,K,y)\le d_r(x,K,y)+2r.
$$
Hence, using the volume estimate \eqref{VE} and optimizing over $r$, we finally obtain, for all $t\in(0,4r_0^2\vee1]$
$$
p(t,x,K,y)\le Ct^{-d/(2\a)}(1+d(x,K,y)^2/t)^{d/(2\a)}
\exp\left\{-\frac{d(x,K,y)^2}{2(1+\d)^2t}+\left(1+\frac 1\d\right)\frac{\|a(\b,\b)\|_\infty t}2\right\}
$$
which is certainly sufficient to imply the claimed asymptotics.
\end{proof}
}

\section{Second variation of the energy in a sub-Riemannian manifold of non-constant rank}\label{SUB}
Let $M$ be a connected $\cinf$ manifold of dimension $d$ and let $a$ be a $\cinf$ non-negative quadratic form on $T^*M$ having a sub-Riemannian structure, as in Section \ref{NPD}. 
In this section and the next, we assume that $M$ is complete for the sub-Riemannian distance.
Since all questions which we address concern properties determined for a finite energy path by any neighbourhood of that path, this assumption of completeness results in no essential loss of generality.
Recall, for a continuous path $\o$ in $M$, the notions of energy and driving path defined in Section \ref{NPD}.
In Section \ref{SRN}, we reviewed the notion of cut locus and defined the Gaussian measure $\mu_\g$ in terms of the bicharacteristic flow.
Extrapolating from the Riemannian case, we might hope to characterize these objects instead in terms of the energy function $I$.
This will be done in Section \ref{QG}.
In the present section, in preparation, we show that the set $\hxy$ of finite-energy paths has, at suitably selected paths $\o$, a well-defined set of tangent directions $\tohxy$, 
and that $\tohxy$ has the structure of a Hilbert space.
Then we show that, when $\o$ is minimal and $\xi$ is a driving path for $\o$, the energy has a well-defined second variation $Q_\xi$ in a dense set of tangent directions, 
which allows us to define a continuous non-negative quadratic form $Q_\xi$ on $\tohxy$.
Finally, we show that $Q_\xi$ is minimized over $\xi$ by a unique driving path $\l$, which is in fact a bicharacteristic.
Given a sub-Riemannian structure $X=(X_1,\dots,X_m)$ for $a$, there are related constructions on the set of control paths, which are well known, evidently depending on the choice of $X$.
We emphasise that $\tohxy$ and $Q=Q_\l$ depend on $a$ alone.

Two sub-Riemannian structures $(X_1,\dots,X_m)$ and $(Y_1,\dots,Y_n)$ are equivalent (see \cite{ABB}) if $X_\ell=\sum_{k=1}^nf_{\ell k}Y_k$
and $Y_k=\sum_{\ell=1}^mg_{k\ell}X_\ell$ for all $\ell$ and $k$, for some $\cinf$ functions $f_{\ell k}$ and $g_{k\ell}$ on $M$.
If $a$ has constant rank, then all sub-Riemannian structures for $a$ are equivalent.
However, this is not true in general, as the following example shows.
In $\R^2$, take 
$$
X_1(x,y)=Y_1(x,y)=y\frac{\pd}{\pd x},\q X_2(x,y)=Y_2(x,y)=\frac{\pd}{\pd y}
$$
\begin{equation}\label{SREX}
X_3(x,y)=\sgn(x)Y_3(x,y)=e^{-1/|x|}\frac{\pd}{\pd x}.
\end{equation}
Then 
$$
\sum_{\ell=1}^3X_\ell\otimes X_\ell=\sum_{\ell=1}^3Y_\ell\otimes Y_\ell
$$ 
but $X$ and $Y$ define inequivalent sub-Riemannian structures on $\R^2$.
Thus, in the non-constant rank case, 
we cannot establish that an object is intrinsic to $a$ by showing it is intrinsic to an equivalence class of sub-Riemannian structures.
Instead, our approach will be to work directly from $a$, using a sub-Riemannian structure for existence but not uniqueness.

Recall that $\hx$ denotes the set of finite-energy paths starting at $x$ and $\hxy$ denotes the set of such paths terminating at $y$.
Also, $\hrm$ denotes the space of absolutely continuous paths $h:[0,1]\ra\R^m$ starting from $0$ such that
$$
\|h\|^2=\int_0^1| \dot h_t|^2dt<\infty.
$$
We fix a sub-Riemannian structure $X=(X_1,\dots,X_m)$ for $a$ and use this to construct some associated objects.
We will make clear which objects depend on the choice of $X$ and which do not.
Given $\o\in H^x$, we define $h(\o)\in\hrm$ by
$$
\dot h_t(\o)=X(\o_t)^*\xi_t
$$
where $\xi$ is any driving path for $\o$.
Then $h(\o)$ does not depend on the choice of $\xi$.
For $x\in M$ and $h\in\hrm$, write $\phi(x,h)$ for the solution $(\phi_t)_{t\in[0,1]}$ of the differential equation
$$
\dot\phi_t=\sum_{\ell=1}^mX_\ell(\phi_t)\dot h_t^\ell,\quad \phi_0=x.
$$
Denote by $p_X(x)$ the orthogonal projection $\R^m\ra(\ker X(x))^\perp$.
For $x\in M$ and $h,k\in\hrm$, define $\pi=\pi(x,h)k\in\hrm$ by
$$
\dot\pi_t=p_X(\phi_t(x,h))\dot k_t.
$$
The following parametrization of $\hx$ by $\hrm$ using $X$ is well known.

\begin{proposition}\label{2.1}
Let $x\in M$ and let $\g\in H^x$. 
Then $\g=\phi(x,h(\g))$.
Moreover, for all $h\in\hrm$, if we set $\o=\phi(x,h)$ and $\pi=\pi(x,h)h$, then $\o\in\hx$ and $\pi=h(\o)$ and $I(\o)=\|\pi\|^2\le\|h\|^2$.
\end{proposition}
\begin{proof}
Since $(\ker X(x))^\perp=\im X(x)^*$, there is a measurable map $\xi:[0,1]\ra\ctm$ over $\o$ such that $\dot\pi_t=X(\o_t)^*\xi_t$.
Then $\dot\o_t=X(\o_t)X(\o_t)^*\xi_t=a(\o_t)\xi_t$ and
$$
I(\o)=\int_0^1\<\xi_t,a(\o_t)\xi_t\>dt=\int_0^1 |X(\o_t)^*\xi_t|^2dt=\|\pi\|^2.
$$
We leave the remaining details to the reader.
\end{proof}

Let $\o$ be a finite-energy path and fix a chart along $\o$.
We say that a driving path $\xi$ for $\o$ is {\em tame} if
$$
\int_0^1|\xi_t|^2dt<\infty
$$
where we have used the Euclidean norm in the chart.
This condition does not depend on the choice of chart.
We say that $\o$ is {\em tame}%
\footnote{%
It is straightforward to see that, if the diffusivity $a$ has constant rank, then every finite-energy path is tame. 
On the other hand, consider on $\R^2$ the vector fields $X_1=\pd/\pd x_1,X_2=x_1\pd/\pd x_2$ and the path $\o_t=(t,t^2/2)$ for $t\in[0,1]$. 
Since $X_1,X_2$ span $\R^2$ except on $\{x_1=0\}$, the only driving path for $\o$ is $\xi_t=dx_1+(1/t)dx_2$. 
Then $\<\xi_t,a(\o_t)\xi_t\>=2$ for all $t$ so $\o$ has finite energy, but $\xi$ is not tame.}
if it has a tame driving path.

Let $\o$ be a tame path in $\hx$.
Fix a chart along $\o$ and a tame driving path $\xi$ for $\o$. 
We will define a space $\tohx$ of {\em finite-energy variations} of $\o$ which, for now, may appear to depend on the choice of chart and of $\xi$. 
Denote by $\tohx$ the set of absolutely continuous maps $v:[0,1]\ra TM$ over $\o$ such that
\begin{equation}\label{v}
\dot v_t=(\nabla_{v_t}a)(\o_t)\xi_t+a(\o_t)\eta_t,\quad v_0=0
\end{equation}
for some measurable path $\eta:[0,1]\ra\ctm$ over $\o$, with
$$
\|v\|^2_\xi:=\int_0^1\<\eta_t,a(\o_t)\eta_t\>dt<\infty.
$$
Note that, in \eqref{v}, the meaning of the derivatives $\dot v_t$ and $\nabla_{v_t}a$ depends on the choice of chart.
Note also that, since $\xi$ is tame, by Gronwall's lemma, there is a constant $C<\infty$ such that $\|v\|_\infty\le C\|v\|_\xi$ for all $v\in\tohx$.
If $\eta$ can be chosen in \eqref{v} so that, in addition,
$$
\int_0^1|\eta_t|^2dt<\infty
$$
then we will call $v$ a {\em tame finite-energy variation}.
Write $\ttohx$ for the set of tame finite-energy variations of $\o$. 
Set $\tohxy=\{v\in\tohx:v_1=0\}$ and $\ttohxy=\{v\in\ttohx:v_1=0\}$.

\begin{proposition}
Let $\o\in \hx$ be tame, with tame driving path $\xi$.
Then $\ttohx$ is dense in $\tohx$.
Moreover $\ttohxy$ is dense in $\tohxy$.
\end{proposition}
\begin{proof}
For $\ve>0$, set $\eta^\ve_t=(\sqrt{a(\o_t)}+\ve I)^{-1}\sqrt{a(\o_t)}\eta_t$ and define $v^\ve$ by
$$
\dot v_t^\ve=(\nabla_{v^\ve_t}a)(\o_t)\xi_t+a(\o_t)\eta_t^\ve,\quad v_0^\ve=0.
$$
Note that $\sqrt{a(\o_t)}(\eta^\ve_t-\eta_t)\to0$ for all $t\in[0,1]$ and 
$$
\<\eta_t^\ve-\eta_t,a(\o_t)(\eta^\ve_t-\eta_t)\>=\ve^2|\eta_t^\ve|^2\le\<\eta_t^\ve,a(\o_t)\eta^\ve_t\>+\ve^2|\eta_t^\ve|^2\le\<\eta_t,a(\o_t)\eta_t\>.
$$
Hence
$$
\int_0^1\<\eta_t^\ve,a(\o_t)\eta_t^\ve\>dt+\ve^2\int_0^1|\eta_t^\ve|^2dt\le\int_0^1\<\eta_t,a(\o_t)\eta_t\>dt<\infty
$$
so $v^\ve\in\ttohx$ for all $\ve$, and by dominated convergence, as $\ve\to0$,
$$
\|v^\ve-v\|^2_\xi=\int_0^1\<\eta_t^\ve-\eta_t,a(\o_t)(\eta^\ve_t-\eta_t)\>dt\to0.
$$
We have proved the first assertion.

Set $V=\{v_1:v\in\tohx\}$. 
There is a finite set $E\sse\tohx$ such that $\{e_1:e\in E\}$ is a basis for $V$.
Given $\d>0$, we can find for each $e\in E$ a tame $\tilde e\in\ttohx$ such that $\|e-\tilde e\|_\xi\le\d$.
Then $|e_1-\tilde e_1|\le C\d$.
We can therefore choose $\d$ sufficiently small that $\{\tilde e_1:e\in E\}$ remains a basis for $V$.
Set
$$
A:=\sup\left\{\left\|\sum_{e\in E}\a_e\tilde e\right\|_\xi:\a_e\in\R,\left|\sum_{e\in E}\a_e\tilde e_1\right|=1\right\}.
$$
Then $A<\infty$.
Given $v\in\tohxy$ and $\ve>0$, there exists $v'\in\ttohx$ such that $\|v-v'\|_\xi<\ve/(1+CA)$.
Then $v'_1\in V$ and $|v'_1|\le C\ve/(1+CA)$.
We can write $v'_1=\sum_{e\in E}\a_e\tilde e_1$ for some $\a_e\in\R$.
Set 
$$
\tilde v=v'-\sum_{e\in E}\a_e\tilde e
$$
then $\tilde v\in\ttohxy$ and 
$$
\|v-\tilde v\|_\xi\le\|v-v'\|_\xi+\left\|\sum_{e\in E}\a_e\tilde e\right\|_\xi\le\ve
$$
which proves the second assertion.
\end{proof}

By standard arguments, the map $\phi:M\times\hrm\ra H$ is differentiable, in both arguments. 
Fix $x\in M$ and $h\in\hrm$.
Set $\o_t=\phi_t(x,h)$ and define
$$
u_t=\phi^*_t(x,h)=\frac{\pd}{\pd x}\phi_t(x,h)\in T_{\o_t}M\otimes\ctxm.
$$
For $k\in\hrm$, define $v=v(k)$ by
$$
v_t=\frac{\pd}{\pd h}\phi_t(x,h)k\in T_{\o_t}M.
$$
Then $u$ and $v$ satisfy the differential equations
\begin{align}\notag
\dot u_t&=\sum_{\ell=1}^m\nabla X_\ell(\o_t)u_t\dot h^\ell_t,\quad u_0=I,\\\label{e1}
\dot v_t&=\sum_{\ell=1}^m\nabla X_\ell(\o_t)v_t\dot h^\ell_t+X_\ell(\o_t)\dot k^\ell_t,\quad v_0=0
\end{align}
where the derivatives $\dot u_t$, $\dot v_t$ and $\nabla X_\ell$ are understood in the chosen chart.
Then, by the variation of constants formula,
\begin{equation}\label{vp}
v_t=\sum_{\ell=1}^mu_t\int_0^t u_s^{-1}X_\ell(\o_s)\dot k_s^\ell ds.
\end{equation}

\begin{proposition}\label{2.2}
Let $\o\in\hx$ be tame and let $k\in\hrm$.
Fix a chart along $\o$ and a tame driving path $\xi$ for $\o$, and write $\tohx$ for the associated space of finite-energy variations.
Set
$$
v(k)=\frac\pd{\pd h}\phi(x,h(\o))k.
$$
Then $v(k)\in\tohx$.
On the other hand, for any $v\in\tohx$, we can define $k(\xi,v)\in\hrm$ by
$$
\dot k_t^\ell(\xi,v)=\<\xi_t,\nabla X_\ell(\o_t)v_t\>+\<\eta_t,X_\ell(\o_t)\>
$$
where $\eta:[0,1]\to T^*M$ is a measurable path over $\o$ satisfying {\rm\eqref{v}}.
Then $k(\xi,v)$ does not depend on the choice of $\eta$, and 
$$
v=\frac\pd{\pd h}\phi(x,h(\o))k(\xi,v).
$$
Moreover there is a constant $C<\infty$, depending only on $I(\o)$, $\int_0^1|\xi_t|^2dt$ and a uniform bound for $X$ and $\nabla X$ along $\o$, such that
$$
\|v(k)\|_\xi\leq C\|k\|,\q \|k(\xi,v)\|\le C\|v\|_\xi.
$$
Moreover, if $v=v(k)$, then $\pi(x,h(\o))(k(\xi,v)-k)=0$.
\end{proposition}
\begin{proof}
Write $h=h(\o)$.
Then $v(k)$ satisfies \eqref{e1} so, by Gronwall's lemma, there is a constant $C<\infty$ such that $\|v(k)\|_\infty\le C\|k\|$.
Here and below, $C$ is understood to have the dependence claimed in the statement.
Since $\xi$ is tame, we can define $g(k)\in\hrm$ by
$$
\dot g_t^\ell(k)=\<\xi_t,\nabla X_\ell(\o_t)v_t(k)\>.
$$
There is a constant $C<\infty$ such that $\|g(k)\|\leq C\|v(k)\|_\infty$.
We can find a measurable map $\eta:[0,1]\ra\ctm$ over $\o$, such that
\begin{equation}\label{e3}
a(\o_t)\eta_t=\sum_{\ell=1}^mX_\ell(\o_t)(\dot k^\ell_t-\dot g^\ell_t(k)).
\end{equation}
Note that
$$
\nabla a(x)=\sum_{\ell=1}^m\nabla X_\ell(x)X_\ell(x)^*+X_\ell(x)\nabla X_\ell(x)^*
$$
so $v(k)$ satisfies 
\begin{equation*}
\dot v_t(k)=(\nabla_{v_t(k)}a)(\o_t)\xi_t+a(\o_t)\eta_t,\quad v_0(k)=0
\end{equation*}
and, moreover
$$
\|v(k)\|^2_\xi=\int_0^1\<\eta_t,a(\o_t)\eta_t\>dt=\|\pi(x,h)(k-g)\|^2\leq C\|k\|^2.
$$
Hence $v(k)\in\tohx$ and $\|v(k)\|_\xi\le C\|k\|$.

On the other hand, for $v\in\tohx$, we see from \eqref{v} that there is a constant $C<\infty$ such that $\|v\|_\infty\leq C\|v\|_\xi$, 
so $\|k(\xi,v)\|\leq C\|v\|_\xi$. Moreover, we can write \eqref{v} in the form
$$
\dot v_t=\sum_{\ell=1}^m\nabla X_\ell(\o_t)v_t\dot h^\ell_t+X_\ell(\o_t)\dot k^\ell_t(\xi,v)
$$
so $v=v(k(\xi,v))$, as claimed.
Finally, if $v=v(k)$, then $v(k(\xi,v)-k)=0$, so, from \eqref{vp}, 
$$
\sum_{\ell=1}^mX_\ell(\o_t)(\dot k^\ell_t(\xi,v)-\dot k^\ell_t)=0
$$ 
for almost all $t$, and so $\pi(x,h)(k(\xi,v)-k)=0$.
\end{proof}

Fix $\o$ and set $\tilde H=\{k\in\hrm:\pi(x,h(\o))k=k\}$. 
Then $\tilde H$ is a closed subspace of the Hilbert space $\hrm$. 
Proposition \ref{2.2} shows that, for any choice of chart along $\o$ and any choice of tame driving path $\xi$ for $\o$,
the map $k\mapsto v(k)$ is a linear isomorphism $\tilde H\to\tohx$, which is bounded with bounded inverse, when $\tohx$ is given the norm $\|.\|_\xi$. 
Hence, the space $\tohx$ does not depend on the choice of chart and driving path.
It clearly does not depend either on any choice of sub-Riemannian structure $X$. 
Moreover, the norms $\|.\|_\xi$ are all equivalent, and all make $\tohx$ into a Hilbert space.

A minimal finite-energy path $\o\in\hxy$ is said to be regular if the linear map
$$
\frac{\pd}{\pd h}\phi_1(x,h(\o)):\hrm\ra\tym
$$
is onto. 
By \eqref{vp}, this is equivalent to Bismut's condition that the deterministic Malliavin covariance matrix
\begin{eqnarray}
\label{mat}
C_1(\o)
=\sum_{\ell=1}^m\int_0^1\left(u_t^{-1}X_\ell(\o_t)\right)\otimes\left(u_t^{-1}X_\ell(\o_t)\right)dt
=\int_0^1u_t^{-1}a(\o_t)(u_t^{-1})^*dt
\end{eqnarray}
is invertible.
In particular, when $a$ is positive-definite, every $\o\in\hxy$ is regular.
In general, these conditions may depend on the choice of sub-Riemannian structure.
By Proposition \ref{2.2}, a tame path $\o\in\hxy$ is regular if and only 
$$
\{v_1:v\in\tohx\}=\tym
$$
so for tame paths the notion of regularity depends only on $a$ and not on the choice of sub-Riemannian structure.
It is straightforward to see that $\o$ is regular if and only if its time-reversal is regular. 
In \cite[pp. 22--24]{B}, Bismut gives an argument which shows that invertibility of $C_1(\o)$ is intrinsic to the equivalence class of the sub-Riemannian structure $X$.
As the example \eqref{SREX} shows, this is not the same as being intrinsic to $a$. 

Consider a regular tame finite-energy path $\o\in\hxy$.
We use the sub-Riemannian structure $X$ to define $h(\o)\in\hrm$ as above and we
write $K$ for the kernel of the linear map $(\pd/\pd h)\phi_1(x,h(\o)):\hrm\to T_yM$.
Then
$$
K^\perp=\{k\in\hrm:\dot k_t^\ell=\<\eta_0,u_t^{-1}X_\ell(\o_t)\>,\eta_0\in\ctxm\}
$$
and $\left.(\pd/\pd h)\phi_1(x,h(\o))\right|_{K^\perp}$ is invertible. 
By the implicit function theorem in Hilbert space, there exist $\d>0$ and a $\cinf$ map $\th:K\to H$ such that, for all $k\in K$, we have
$$
\phi_1(x,h(\o)+k+\th(k))=y
$$
and such that, for all $k\in K$ and all $k'\in K^\perp$ with $\| k+k'\|<\d$, we have $\th(k)\in K^\perp$ and
$$
\phi_1(x,h+k+k')=y\quad\text{only if}\quad k'=\th(k).
$$
Note that $\th(0)=0$. 
For $k\in K$ and $\ve$ sufficiently small, we have
\begin{equation*}
\phi_1(x,h+\ve k+\th(\ve k))=y.
\end{equation*}
On differentiating in $\ve$ at $0$, we obtain $(\pd/\pd h)\phi_1(x,h)(k+\th'(0)k)=0$, so $\th'(0)k\in K$.
Since $\th$ takes values in $K^\perp$, we deduce that $\th'(0)=0$.
On taking the second derivative, we obtain
\begin{equation}\label{kk}
\frac{\pd^2}{\pd h^2}\phi_1(x,h)(k,k)
+\frac{\pd}{\pd h}\phi_1(x,h)\th''(0)(k,k)=0.
\end{equation}
Since $\th''(0)(k,k)\in K^\perp$, this equation determines $\th''(0)$.

We note the following useful identity. 
Let $k\in K^\perp$ and let $k'\in\hrm$. 
Then $\dot k_t^\ell=\<\eta_0,u_t^{-1}X_\ell(\o_t)\>$ for some $\eta_0\in\ctxm$.
So
\begin{equation}\label{ui}
\<k,k'\>
=\sum_{\ell=1}^m\int_0^1\<\eta_0,u_t^{-1}X_\ell(\o_t)\>(\dot k_t')^\ell\, dt=\<\eta_0,u_1^{-1}v_1(k')\>.
\end{equation}

The key arguments and computations, for a sub-Riemannian structure $X$, in the next two results are due to Bismut \cite[Theorem 1.17 and Theorem 1.24]{B}.
Our new contribution is to construct objects and show results which do not depend on the choice of $X$.
\begin{proposition}\label{2.3}
Let $\o\in\hxy$ be tame and regular and let $\xi$ be a tame driving path for $\o$.
Let $v\in\ttohxy$ be a tame finite-energy variation.
There exists a measurable map
$$
(\ve,t)\mapsto\xi^\ve_t:(-1,1)\times[0,1]\ra\ctm
$$
such that $\xi^0=\xi$ and
\begin{itemize}
\item[{\rm (i)}] $\o^\ve=\pi\xi^\ve\in\hxy$ for all $\ve$, with $\dot\o_t^\ve=a(\o_t^\ve)\xi_t^\ve$,
\item[{\rm (ii)}] in any chart along $\o$, there is a constant $C<\infty$ such that, for all $\ve\in(-1,1)$,
$$
\sup_{t\in[0,1]}|\o_t^\ve-\o_t-\ve v_t|\leq C\ve^2
$$
and, writing $\xi^\ve=(\o^\ve,p^\ve)$ and $ \eta=(\o,q)$, 
$$
\int_0^1\<p_t^\ve-p_t-\ve q_t,a(\o_t^\ve)(p_t^\ve-p_t-\ve q_t)\>dt\leq C\ve^4.
$$
\end{itemize}
Moreover, for any such map $\ve\mapsto \xi^\ve$, the map $\ve\mapsto I(\o^\ve)$ is differentiable at $\ve=0$, and is twice differentiable at $\ve=0$ if $\o$ is minimal.
Define
$$
L(v)=\left.\frac{\pd}{\pd\ve}\right|_{\ve=0} I(\o^\ve)
$$
then $L$ extends uniquely to a continuous linear form on $\tohxy$.
In the case where $\o$ is minimal, define
\begin{equation}\label{Q}
Q_\xi(v)=\frac{1}{2}\left.\frac{\pd^2}{\pd\ve^2}\right|_{\ve=0} I(\o^\ve)
\end{equation}
then $Q_\xi$ extends uniquely to a continuous quadratic form on $\tohxy$.
Finally, given a sub-Riemannian structure $X$ for $a$, for $h=h(\o)$ and
$k=k(\xi,v)$, we have
$$
L(v)=2\<h,k\>
$$
and, when $\o$ is minimal, then $h\in K^\perp$ and
$$
Q_\xi(v)=\|k\|^2+\<h,\th''(0)(k,k)\>.
$$
\end{proposition}
\begin{proof}
Fix a sub-Riemannian structure $X$ for $a$.
We will use $X$ to construct a map $\xi^\ve$ having the claimed properties. 
Set $h=h(\o)$ and $k=k(\xi,v)$.
Since $v_1=0$, we have $k\in K$.
For $\ve\in(-1,1)$, set $h^\ve=h+\ve k+\th(\ve k)$ and set $\o^\ve=\phi(x,h^\ve)$.
Then $\o^\ve\in\hxy$ for all $\ve$.
Moreover, the map $\ve\mapsto\o^\ve:(-1,1)\to\O$ is $C^1$ and there is a constant $C<\infty$ such that
$$
\sup_{t\in[0,1]}|\o_t^\ve-\o_t-\ve v_t|\leq C\ve^2.
$$
Moreover, fixing a chart along $\o$ and writing $\xi=(\o,p)$ and $\eta=(\o,q)$, we have 
$$
\dot\o_t^\ve-a(\o_t^\ve)(p_t+\ve q_t)=X_\ell(\o_t^\ve)\dot g_t^{\ve,\ell}
$$
where
\begin{align*}
\dot g_t^{\ve,\ell}&=\dot h_t^\ell+\ve \dot k_t^\ell+\dot\th_t^\ell(\ve k)
-\<p_t+\ve q_t,X_\ell(\o_t^\ve)\>\\
&=-\left\<p_t,X_\ell(\o_t^\ve)-X_\ell(\o_t)-\ve\nabla X_\ell(\o_t)v_t\right\>
-\ve\<q_t,X_\ell(\o_t^\ve)-X_\ell(\o_t)\>+\dot\th_t^\ell(\ve k).
\end{align*}
Since $\xi$ and $\eta$ are tame, there is a constant $C<\infty$ such that, for all $\ve$, we have
$$
\int_0^1 |\dot g_t^\ve|^2dt\leq C\ve^4
$$
so we can find a measurable map $r_t^\ve$ such that
$$
X_\ell(\o_t^\ve)\dot g_t^{\ve,\ell}=a(\o_t^\ve)r_t^\ve
$$
and
$$
\int_0^1\<r_t^\ve,a(\o_t^\ve)r_t^\ve\>dt\leq C\ve^4.
$$
If we now set $p_t^\ve=p_t+\ve q_t+r_t^\ve$, then $\xi_t^\ve=(\o_t^\ve,p_t^\ve)$
has the required properties.

Suppose now, more generally, that $\xi^\ve=(\o^\ve,p^\ve)$ and $\eta=(\o,q)$ are maps having the properties described in the statement. 
Define $h^\ve\in\hrm$ by
$$
\dot h_t^{\ve,\ell}=\<p_t^\ve,X_\ell(\o_t^\ve)\>.
$$
Then
\begin{align*}
&\dot h_t^{\ve,\ell}-\dot h_t^\ell-\ve \dot k_t^\ell\\
&\quad=\<p_t^\ve-p_t-\ve q_t,X_\ell(\o_t^\ve)\>+\ve\<q_t,X_\ell(\o_t^\ve)-X_\ell(\o_t)\>+\<p_t,X_\ell(\o_t^\ve)-X_\ell(\o_t)-\ve\nabla X_\ell(\o_t)v_t\>
\end{align*}
so, since $\xi$ and $\eta$ are tame, there is a constant $C<\infty$ such that, for all $\ve$, we have
$$
\|h^\ve-h-\ve k\|\leq C\ve^2.
$$
Write $k^\ve$ for the orthogonal projection of $\ve^{-1}(h^\ve-h)$ onto $K$. 
Since $k\in K$, we have $\|k^\ve-k\|\le\|\ve^{-1}(h^\ve-h)-k\|\le C\ve$. 
Since $\phi_1(x,h^\ve)=y$, we have, for $\ve$ sufficiently small,
$$
h^\ve=h+\ve k^\ve+\th(\ve k^\ve).
$$
Hence, as $\ve\ra 0$,
$$
I(\o^\ve)-I(\o)=\|h+\ve k^\ve+\th(\ve k^\ve)\|^2-\|h\|^2
=2\ve\<h,k\>+O(\ve^2).
$$
Hence $\ve\mapsto I(\o^\ve)$ is differentiable at $\ve=0$ with derivative $L(v)=2\<h,k\>$.
By Proposition \ref{2.2}, we have $\|k(\xi,v)\|\le C\|v\|_\xi$, so $L$ is continuous on $\ttohxy$, which is dense in $\tohxy$, so $L$ extends uniquely to $\tohxy$.

Now, if $\o$ is minimal, we must have $L(v)=0$ for all $v\in\ttohxy$ and hence for all $v\in\tohxy$.
So, for any $k'\in K$, we have
$$
\<h,k'\>=\<h,k(\xi,v(k'))\>=L(v(k'))=0
$$
and so $h\in K^\perp$.
Then, for $\o^\ve$ as above, as $\ve\ra 0$, we have
$$
I(\o^\ve)-I(\o)=\ve^2\|k^\ve\|^2+2\<h,\th(\ve k^\ve)\>+\|\th(\ve k^\ve)\|^2=\ve^2\{\|k\|^2+\<h,\th''(0)(k,k)\>\}+O(\ve^3).
$$
This shows that $\ve\mapsto I(\o^\ve)$ is twice differentiable at $\ve=0$ with the claimed second derivative, which is then continuous on $\ttohxy$
and extends uniquely to $\tohxy$.
\end{proof}

Note that, since $L(v)$ and $Q_\xi(v)$ can be computed either by choosing a suitable family of paths $\ve\mapsto\xi^\ve$ or by choosing a sub-Riemannian structure $X$, they depend on neither choice.
For $\g\in\hxy$ tame, regular and minimal, we define a continuous quadratic form $Q$ on $\tghxy$ by
$$
Q(v)=\inf_{\xi}Q_\xi(v)
$$
where the infimum is taken over all tame driving paths $\xi$.

\begin{proposition}\label{2.4}
Let $\g\in H^{x,y}$ be tame, regular and minimal. 
Then there exists a unique tame driving path $\l$ such that $Q=Q_\l$. 
The path $\l$ is a bicharacteristic and is the only bicharacteristic which is a driving path for $\g$.
Moreover, given a sub-Riemannian structure $X$ for $a$, we have
$$
Q_\xi(v)=Q(v)+\|k(\xi,v)-k(\l,v)\|^2,\q  Q(v)=q(k(\l,v))
$$
where $q$ is the quadratic form on $K=\ker(\pd/\pd h)\phi_1(x,h(\g))$ given by
$$
q(k)=\|k\|^2-\left\<\l_1,\frac{\pd^2}{\pd h^2}\phi_1(x,h(\g))(k,k)\right\>.
$$
\end{proposition}
\begin{proof}
Choose a sub-Riemannian structure $X$ for $a$ and define as above
$$
h=h(\g),\q u_t=\frac\pd{\pd x}\phi_t(x,h),\q K=\ker\frac\pd{\pd h}\phi_1(x,h).
$$
Then, by Proposition \ref{2.3}, we have $h\in K^\perp$, so there exists a unique $\l_0\in T^*_xM$ such that
$$
\dot h_t^\ell=\<\l_t,X_\ell(\g_t)\>
$$
for almost all $t$, where $\l_t=(u_t^{-1})^*\l_0$. 
Fix a chart along $\g$ and write $\l_t=(\g_t,p_t)$.
Then
$$
\dot\g_t=\sum_{\ell=1}^mX_\ell(\g_t)\dot h_t^\ell,\q \dot p_t=-\sum_{\ell=1}^m\<p_t,\nabla X_\ell(\g_t)\>\dot h_t^\ell
$$
so $\l$ is a bicharacteristic.
Suppose, on the other hand, that $\xi=(\g,q)$ is a bicharacteristic over $\g$ and define $k\in\hrm$ by $\dot k_t^\ell=\<\xi_t,X_\ell(\g_t)\>$.
Then $\g=\phi(x,k)$ and $I(\g)=\|k\|^2$ and
$$
\dot q_t=-\sum_{\ell=1}^m\<q_t,\nabla X_\ell(\g_t)\>\dot k_t^\ell
$$
so $k=h(\g)$ and $q=p$, proving uniqueness.

Write $h$ for $h(\g)$ and recall that $\dot h_t^\ell=\<\l_0,u_t^{-1}X_\ell(\g_t)\>$ and $\l_1=(u_1^{-1})^*\l_0$.
Then, by \eqref{kk} and \eqref{ui}, we have
\begin{equation}\label{ix}
q(k)=\|k\|^2+\<h,\th''(0)(k,k)\>.
\end{equation}
So $q(k(\xi,v))=Q_\xi(v)$ by Proposition \ref{2.3}.
Take now $v=v(k)$ and set $k'=k-k(\l,v)$.
Then $(\pd/\pd h)\phi(x,h)k'=0$, so $\sum_{\ell=1}^mX_\ell(\g_t)(\dot k_t')^\ell=0$ for almost all $t$, and so $\phi(x,h+\ve k')=\g$ for all $\ve\in\R$. 
Hence $q(k')=\|k'\|^2$ and it will suffice to show that $q(k(\l,v),k')=0$.
Recall that $\dot k_t^\ell(\l,v)=\<\l_t,\nabla X_\ell(\g_t)v_t\>+\<\eta_t,X_\ell(\g_t)\>$.
We differentiate the identity
\begin{equation}
\label{k}
\left(\frac{\pd}{\pd x}\phi_1(x,h)\right)^{-1}
\frac{\pd}{\pd h}\phi_1(x,h)k'
=\sum_{\ell=1}^m\int_0^1\left(\frac{\pd}{\pd x}\phi_t(x,h)\right)^{-1}X_\ell(\phi_t(x,h))(\dot k_t')^\ell\, dt
\end{equation}
in $h$, in the direction $k(\l,v)$, to obtain
$$
u_1^{-1}\frac{\pd^2}{\pd h^2}\phi_1(x,h)(k(\l,v),k')
=\sum_{\ell=1}^m\int_0^1 u_t^{-1}\nabla X_\ell(\g_t)v_t(\dot k_t')^\ell\, dt.
$$
Hence
$$
\left\<\l_1,\frac{\pd^2}{\pd h^2}\phi_1(x,h)(k(\l,v),k')\right\>
=\sum_{\ell=1}^m\int_0^1\<\l_t,\nabla X_\ell(\g_t)v_t\>(\dot k_t')^\ell\, dt=\<k(\l,v),k'\>.
$$
This shows that $q(k(\l,v),k')=0$ as required.
\end{proof}

Note that our extra condition of tameness, under which regularity of finite-energy paths is intrinsic to $a$, does not exclude any normal minimal paths as these all have smooth driving paths.

\section{Characterization of the cut locus and fluctuation measure by the energy function}\label{QG}
Throughout this section, $\g$ is a regular tame minimal path in $H^{x,y}$, $\l$ is the unique bicharacteristic projecting to $\g$ and $Q$ is the quadratic form on $\tghxy$ defined in the preceding section. 
Fix $s\in[0,1]$ and $\eta_s\in\ctgsm$. 
For $t\in[s,1]$, define $J_{ts}:\ctgsm\ra\tgtm$ by
$$
J_{ts}\eta_s
=\left.\frac{\pd}{\pd\ve}\right|_{\ve=0}\pi\psi_{t-s}(\l_s+\ve \eta_s)
$$
where $(\psi_t)_{t\in\R}$ is the bicharacteristic flow.
We already defined $J_t=J_{t0}$ by equation \eqref{JFE} above.
Set $v_t=J_{ts}\eta_s$.
Given a sub-Riemannian structure $X$ for $a$, the following equations determine $\phi_{ts}(x,h)$ and $\dot h_t$, for $t\in[s,1]$ and $x$ near $\g_t$, 
as functions of $\l_s\in\ctgsm$
$$
\phi_{ss}(x,h)=x,\q
\dot\phi_{ts}(x,h)=\sum_{\ell=1}^mX_\ell(\phi_{ts}(x,h))\dot h_t^\ell,\q
\dot h_t^\ell=\left\<\l_s,\left(\frac\pd{\pd x}\phi_{ts}(\g_s,h)\right)^{-1}X_\ell(\phi_{ts}(\g_s,h))\right\>.
$$
On differentiating in $\l_s$ in the direction $\eta_s$, we find that $v_t=(\pd/\pd h)\phi_{ts}(\g_s,h)k$, where $k$ satisfies, for $t\in[s,1]$,
\begin{equation}\label{ak}
\dot k_t^\ell =\<\eta_s,u_{ts}^{-1}X_\ell(\g_t)\>+A_{ts}^\ell(\l,k).
\end{equation}
Here, $u_{ts}=(\pd/\pd x)\phi_{ts}(\g_s,h)=u_tu_s^{-1}$ and
$$
A_{ts}^\ell(\l,k)
=\left\<\l_s,\frac{\pd}{\pd h}\left[\left(\frac{\pd}{\pd x}\phi_{ts}(\g_s,h) \right)^{-1}X_\ell(\phi_{ts}(\g_s,h))\right]k\right\>.
$$
By writing a differential equation for $A_{ts}^\ell(\l,k)$, we find a constant $C<\infty$ such that, for all $t\in[s,1]$,
$$
|A_{ts}(\l,k)|^2\leq C\int_s^t|\dot k_r|^2\, dr.
$$
It follows that the equation \eqref{ak} uniquely determines $k$.
Note that, if we set $v_t=0$ and $\dot k_t=0$ for $t\in[0,s)$, then we have, for all $t\in[0,1]$,
\begin{equation}\label{al}
v_t=\frac\pd{\pd h}\phi_t(x,h)k, \q 
\dot k_t^\ell =\<u_s^*\eta_s,u_t^{-1}X_\ell(\g_t)\>1_{\{t\ge s\}}+A_{t0}^\ell(\l,k).
\end{equation}

\begin{proposition}\label{2.5}
Let $\g\in\hxy$ be tame, regular and minimal and let $v\in\tghxy$.
Then $Q(v)=0$ if and only if $v_t=J_t\eta_0$ for all $t\in[0,1]$, for some $\eta_0\in\ctxm$.
\end{proposition}
\begin{proof}
Choose a sub-Riemannian structure $X$ for $a$.
Suppose that $Q(v)=0$.
Set $k=k(\l,v)$, then $q(k)=0$. 
On differentiating the identity \eqref{k} in $h$, in the direction $k$, we obtain
$$
u_1^{-1}\frac{\pd^2}{\pd h^2}\phi_1(x,h)(k,k')
=\sum_{\ell=1}^m\int_0^1 (\dot k_t')^\ell\frac{\pd}{\pd h}\left[\left(\frac{\pd}{\pd x}\phi_t(x,h) \right)^{-1}X_\ell(\phi_t(x,h))\right]k\, dt
$$
so, by Proposition \ref{2.4},
\begin{equation}\label{DIFF}
q(k,k')=\sum_{\ell=1}^m\int_0^1 (\dot k_t')^\ell\left\{\dot k_t^\ell
-\left\<\l_0,\frac{\pd}{\pd h}\left[\left(\frac{\pd}{\pd x}\phi_t(x,h) \right)^{-1}X_\ell(\phi_t(x,h))\right]k\right\>\right\}\, dt.
\end{equation}
Since $\g$ is minimal, $q$ is non-negative on $K$, so $q(k,k')=0$ for all $k'\in K$.
Hence there exists an $\eta_0\in\ctxm$ such that
$$
\dot k_t^\ell=\<\eta_0,u_t^{-1}X_\ell(\g_t)\>+A_{t0}^\ell(\l,k).
$$
As we argued above, this forces $v_t=J_t\eta_0$ for all $t$.
On the other hand, if $v_t=J_t\eta_0$ then, by the same calculations, we see that $v=(\pd/\pd h)\phi(x,h)k$ where $k$ satisfies $q(k,k')=0$ for all $k'\in K$.
But $k\in K$, so $Q(v)\le q(k,k)=0$.
\end{proof}

\begin{proposition}\label{2.6}
Let $\g\in\hxy$ be tame, regular and minimal and let $\eta_0\in \ctxm$.
Then $J_t\eta_0=0$ for all $t$ only if $\eta_0=0$.
\end{proposition}
\begin{proof}
Fix a sub-Riemannian structure $X$ for $a$ and a chart along $\g$.
Write $\psi_t(\l_0+\ve\eta_0)=(\g^\ve_t,p^\ve_t)$ and $\dot h^{\ve,\ell}_t=\<p_t^\ve,X_\ell(\g^\ve_t)\>$.
Set 
$$
v_t=J_t\eta_0=\left.\frac\pd{\pd\ve}\right|_{\ve=0}\g^\ve_t,\q
r_t=\left.\frac\pd{\pd\ve}\right|_{\ve=0}p^\ve_t,\q
\dot k^\ell_t=\left.\frac\pd{\pd\ve}\right|_{\ve=0}\dot h^{\ve,\ell}_t.
$$
By differentiating the bicharacteristic equations we obtain, for all $t\in[0,1]$,
$$
\dot v_t=\sum_{\ell=1}^m\nabla X_\ell(\g_t)v_t\dot h_t^\ell+X_\ell(\g_t)\dot k_t^\ell,\q
\dot k_t^\ell=\<r_t,X_\ell(\g_t)\>+\<p_t,\nabla X_\ell(\g_t)v_t\>.
$$
But $v_t=J_t\eta_0=0$ for all $t$ so $\sum_{\ell=1}^mX_\ell(\g_t)\dot k_t^\ell=0$ and so $|\dot k_t|^2=\sum_{\ell=1}^m\<r_t,X_\ell(\g_t)\>^2=0$ for all $t$.
Now 
\begin{equation}\label{ELL}
\dot h_t^{\ve,\ell}=\<p_t^\ve,X_\ell(\g_t^\ve)\>=\<\l_0+\ve\eta_0,(u_t^\ve)^{-1}X_\ell(\g_t^\ve)\>
\end{equation}
where 
$$
\dot u_t^\ve=\nabla X_\ell(\g_t^\ve)u_t^\ve\dot h^{\ve,\ell}_t,\q u_0=\id. 
$$
By differentiating this equation, we see that $(\pd/\pd\ve)|_{\ve=0}u_t^\ve=0$, so on differentiating \eqref{ELL} we obtain
$$
\<\eta_0,u_t^{-1}X_\ell(\g_t)\>=0
$$
for all $t$, and this implies that $\eta_0=0$ since $\g$ is regular.
\end{proof}

The next result is a characterization of the cut locus by path properties intrinsic to $a$.
An analogous characterization in terms of a given sub-Riemannian structure is given in \cite[Th\'eor\`eme 1.18]{BA}. 
We note the use of Proposition \ref{2.6}. 
\begin{theorem}\label{2.7}
Let $x,y\in M$. The following are equivalent
\begin{itemize}
\item[{\rm (i)}] $(x,y)$ lies outside the cut locus,
\item[{\rm (ii)}] there is a unique minimal path $\g\in\hxy$, which is tame and regular, and the quadratic form $Q$ is positive-definite on $\tghxy$.
\end{itemize}
\end{theorem}
\begin{proof}
In both (i) and (ii) we have a unique minimal path $\g\in\hxy$.
By Proposition \ref{2.4}, if $\g$ is tame and regular then $\g$ is the projection of a bicharacteristic, while any projected bicharacteristic is tame.

Suppose that $Q$ is positive-definite and that $J_1\eta_0=0$ for some $\eta_0\in\ctxm$.
Then $v_t=J_t\eta_0\in\tghxy$ and $Q(v)=0$, so $J_t\eta_0=0$ for all $t$, and so $\eta_0=0$, by Proposition \ref{2.6}.
Hence $J_1$ is invertible.

Suppose, on the other hand, that $J_1$ is invertible. 
Since $\{v_1:v\in\tghx\}$ contains $\{J_1\eta_0:\eta_0\in\ctxm\}$, we see that $\g$ is regular.
Suppose further that $Q(v)=0$ for some $v\in\tghxy$.
By Proposition \ref{2.5}, we have $v_t=J_t\eta_0$ for some $\eta_0\in\ctxm$.
Then $J_1\eta_0=0$, so $\eta_0=0$, and so $v=0$.
Hence $Q$ is positive-definite.
\end{proof}

Recall from \eqref{JFE} that we define $K_t:\ctym\ra\tgtm$ by
$$
K_t\eta_1
=\left.\frac{\pd}{\pd\ve}\right|_{\ve=0}\pi\psi_{-(1-t)}(\l_1-\ve\eta_1)
$$
and, for $s,t\in[0,1]$ with $s\le t$, set
$$
C(s,t)=C(t,s)^*=J_sJ_1^{-1}K_t^*\in\tgsm\otimes\tgtm.
$$

\begin{proposition}\label{2.8}
Let $\g\in\hxy$ be tame, regular and minimal.
Suppose that $x$ and $y$ are non-conjugate along $\g$.
Let $s\in[0,1]$ and $\b\in\ctgsm$.
For $t\in[0,1]$, set $v_t^{\b,s}=C(t,s)\b$.
Then $v^{\b,s}\in\tghxy$ and, for all $v\in\tghxy$,
$$
Q(v,v^{\b,s})=\<\b,v_s\>.
$$
\end{proposition}
\begin{proof}
Set $\eta_0=J_1^{-1}J_{1s}\b$.
We can define $w=w^1-w^2\in\tghxy$ by setting $w^1_t=J_t\eta_0$ for all $t$, and $w^2_t=0$ for $t\le s$ and $w^2_t=J_{ts}\b$ for $t>s$.
Then $w=(\pd/\pd h)\phi(x,h)k$, where $k=k^1-k^2$ with
$$
(\dot k^1_t)^\ell=\<\eta_0,u_t^{-1}X_\ell(\g_t)\>+A_{t0}^\ell(\l,k^1),\q
(\dot k^2_t)^\ell=\<u^*_s\b,u^{-1}_tX_\ell(\g_t)\>1_{\{t\ge s\}}+A_{t0}^\ell(\l,k^2).
$$
Take $v'\in\tghxy$ and set $k'=k(\l,v')$.
Then, using \eqref{DIFF} for the second equality, we have
\begin{align*}
Q(w,v')
&=q(k,k')=\sum_{\ell=1}^m\int_0^1(\dot k_t')^\ell\left\{\dot k_t^\ell-A_{t0}^\ell(\l,k) \right\}\, dt\\
&=\sum_{\ell=1}^m\int_0^1\<\eta_0,u_t^{-1}X_\ell(\g_t)\>(\dot k_t')^\ell\, dt-\int_s^1\<u^*_s\b,u_t^{-1}X_\ell(\g_t)\>(\dot k_t')^\ell\, dt=\<\b,v_s'\>.
\end{align*}
For $t\leq s$ we have
$$
w_t=J_{t0}J_{10}^{-1}J_{1s}\b=J_tJ_1^{-1}K_s^*\b=C(t,s)\b.
$$
Consider now the analogous construction in reverse time.
Set $\eta_1=K_0^{-1}K_{0s}\b$ and define $\hat w=\hat w^1-\hat w^2\in\tghxy$ by setting $\hat w_t^1=K_t\eta_1$,
and $\hat w^2_t=0$ for $t\ge s$ and $\hat w^2_t=K_{ts}\b$ for $t<s$.
Then, by the same argument, $Q(\hat w,v')=\<\b,v'_s\>$ for all $v'\in\tghxy$.
But $w-\hat w\in\tghxy$, so this implies $Q(w-\hat w,w-\hat w)=0$ and hence $\hat w=w$.
So, for $t\ge s$, we have
$$
w_t=\hat w_t=K_{t1}K_{01}^{-1}K_{0s}\b=K_t(J_1^{-1})^*J_s^*\b=C(s,t)^*\b.
$$
Hence $w=v^{\b,s}$ and $v^{\b,s}$ has the claimed property.
\end{proof}

\begin{theorem}\label{SRMR}
Let $x,y\in M$ and suppose that $(x,y)$ lies outside the cut locus.
Denote the unique minimal path in $\hxy$ by $\g$.
Recall the definitions \eqref{BWW} of the process $W$ and \eqref{SDEF} of the random variables $S(z)$ for $z\in\R^d$.
Then $\E(e^{pS(z)/2})<\infty$ for all $z\in\R^d$, for some $p>1$.

Define a new probability measure $\tilde\PP$ on $\orm$ by $d\tilde\PP/d\PP\propto e^{S(0)/2}$ 
and write $\mg$ for the law on $\tgoxy$ of $Y(0)=v(W)$ under $\tilde\PP$.
Then $\mg$ is a zero-mean Gaussian probability measure on $\tgoxy$ with the following properties
\begin{itemize}
\item[{\rm (i)}] for all continuous linear functionals $\phi$ on $\tgoxy$, 
$$
\int_\tgoxy\phi(v)^2\mg(dv)=Q(\tilde\phi)
$$
where $\tilde\phi\in\tghxy$ is given by $\phi(v)=Q(\tilde\phi,v)$ for all $v\in\tghxy$,
\item[{\rm (ii)}] for all $s,t\in[0,1]$ with $s\le t$,
$$
\int_{\tgoxy} v_s\otimes v_t\,\mg(dv)=J_sJ_1^{-1}K_t^*.
$$
\end{itemize}
Moreover, the properties {\rm (i)} and {\rm (ii)} both characterize $\mg$ uniquely.
\end{theorem}
\begin{proof}
Choose a sub-Riemannian structure $X$ for $a$ and consider the continuous quadratic form on $\hrm$ given by
$$
s(k)=\<\l_1,(\pd/\pd h)^2\phi_1(x,h)(k,k)\>.
$$
Fix an orthonormal basis $(e_n:n\in\N)$ for $K$ which diagonalizes $s$ on $K$.
Formally, we have $S(z)=s(W(z))$ so, by a standard calculus for the Wiener chaos, 
\begin{equation}\label{TSZ}
S(z)=\sum_ns(e_n)\<W,e_n\>^2+2\sum_ns(e_n,\t(z))\<W,e_n\>+s(\t(z))
\end{equation}
where the sums are understood in $L^2(\PP)$ and the random variables $\<W,e_n\>$ are independent standard Gaussians.
Set $s_n=-s(e_n)$ and $\s_n=s(e_n,\t(z))$.
We see from \eqref{TSZ} that the series $\sum_ns_n$ converges and $\sum_n\s_n^2<\infty$. In particular $s_n\to0$ as $n\to\infty$.
Since $x$ and $y$ are non-conjugate along $\g$, for all non-zero $k\in K$, we have
$$
\|k\|^2-s(k)=q(k)=Q(v(k))+\|k'\|^2>0
$$
where $k'=k-k(\l,v(k))$.
Hence $1+s_n>0$ for all $n$.
Denote by $\nu$ the standard Gaussian distribution on $\R$.
Then, for all $s\in(-1,\infty)$ and all $\s\in\R$, we have
\begin{equation}\label{TSB}
\int_\R e^{\s x}e^{-sx^2/2}\nu(dx)=\frac{1}{\sqrt{1+s}}e^{\s^2/(2(1+s))}.
\end{equation}
Hence, for $p>1$ sufficiently close to $1$, for all $z\in\R^d$, we have
$$
\E(e^{pS(z)/2})=e^{ps(\t(z))}\prod_n\int_\R e^{p\s_n x}e^{-ps_nx^2/2}\nu(dx)=e^{s(\t(z))}\prod_n\frac{1}{\sqrt{1+ps_n}}e^{p^2\s_n^2/(2(1+ps_n))}<\infty.
$$ 
Let $\phi$ be a continuous linear functional on $\tgoxy$ and write $\tilde\phi=v(k^*)$, where $k^*=k(\l,\tilde\phi)$.
Then
$$
\phi(v(e_n))=Q(\tilde\phi,v(e_n))=q(k^*,e_n)=\<k^*,e_n\>q(e_n).
$$
From \eqref{TSB}, we see that under $\tilde\PP$, the random variables $\<W,e_n\>$ are independent zero-mean Gaussians of variance $(1+s_n)^{-1}=q(e_n)^{-1}$.
Now
$$
\phi(Y(0))=\phi(v(W))=\sum_n\<W,e_n\>\phi(v(e_n))
$$
so the law $\mg$ of $Y(0)$ on $\tgoxy$ under $\tilde\PP$ is a zero-mean Gaussian measure with 
$$
\int_\tgoxy\phi(v)^2\mg(dv)=\E(\phi(Y(0))^2)=\sum_n\phi(v(e_n))^2/q(e_n)=\sum_n\<k^*,e_n\>^2q(e_n)=q(k^*)=Q(\tilde\phi)
$$
as claimed.
Finally, by Proposition \ref{2.8}, for $\a\in\ctgsm$ and $\b\in\ctgtm$, we have
$$
\int_\tgoxy\<\a,v_s\>\<\b,v_t\>\mg(dv)=Q(v^{\a,s},v^{\b,t})=\<\a,C(s,t)\b\>
$$
so $\mg$ also has the claimed covariance. The uniqueness statements are standard.
\end{proof}

\section{Brownian bridge on a Riemannian manifold}\label{RIEM}
Suppose now that $M$ is a connected $\cinf$ Riemannian manifold.
In this section we will use the Levi--Civita connection $\nabla$ and Riemann curvature tensor $R$ to give alternative characterizations for the cut locus and for the limit Gaussian measures of small-time Brownian bridges.
Write $\Delta$ for the Laplace--Beltrami operator and set 
$$
\cL=\tfrac12\Delta+\bar X_0
$$
where $\bar X_0$ is a $\cinf$ vector field on $M$. 
Then $\cL$ has the form \eqref{LAB} and its diffusivity $a$ is the inverse of the metric tensor.
We obtain in this way all operators $\cL$ of the form \eqref{LAB} with $a$ everywhere positive-definite.
In this context, it is well known that there exist $m\in\N$ and $C^\infty$ vector fields $X_1,\dots,X_m$ on $M$ such that \eqref{LAC} and \eqref{LAD} hold, with
$$
\spann\{Y(x):Y\in\cA(X_1,\dots,X_m)\}=\spann\{X_1(x),\dots,X_m(x)\}=T_xM\q\text{for all $x\in M$.}
$$
In the case $\bar X_0=0$, the measure $\mexy$ is the law of the Riemannian Brownian bridge from $x$ to $y$ of speed $\sqrt\ve$.

Fix $x,y\in M$ and assume, as above, that there is a unique minimal path $\g\in\hxy$, which is strongly minimal.
It is well known that in this context $\g$ is always the projection of a bicharacteristic.
Define a linear map $R_t:\tgtm\to\tgtm$, symmetric with respect to the metric, by $R_t=R(.,\dot\g_t)\dot\g_t$. 
It is straightforward to see that the set $\tghxy$, defined more generally in Section \ref{SUB},
is here given by the set of all absolutely continuous paths $v$ in $\tgoxy$ such that
$$
\int_0^1|\nabla v_t|^2dt<\infty.
$$
Moreover, we can define, for $\eta\in\R$ sufficiently small, a path $\g^\eta\in\hxy$ by $\g^\eta_t=\exp_{\g_t}(\eta v_t)$. Then the map $\eta\mapsto I(\g^\eta)$ is twice differentiable near $0$ with
\begin{equation}\label{QR}
Q(v)=\left.\frac{\pd^2}{\pd\eta^2}\right|_{\eta=0}I(\g^\eta)=\int_0^1|\nabla v_t|^2\, dt-\int_0^1\<v_t,R_tv_t\>\, dt.
\end{equation}
See, for example \cite{K}.
By a standard calculation, the processes $(J_t)_{t\in[0,1]}$ and $(K_t)_{t\in[0,1]}$, defined above, are Jacobi fields along $\g$ and satisfy the differential equations
$$
\nabla^2J_t+R_tJ_t=0,\quad J_0=0,\quad\nabla J_0=a(x)
$$
and
$$
\nabla^2K_t+R_tK_t=0,\quad K_1=0,\quad\nabla K_1=-a(y).
$$
Let $(b_t)_{t\in[0,1]}$ be a Brownian motion in $\txm$, starting from $0$.
Set $z_t=b_t-tb_1$. 
Then $(z_t)_{t\in[0,1]}$ is a Brownian bridge in $\txm$ from $0$ to $0$ in time $1$. 
Let $(\t_t)_{t\in[0,1]}$ denote parallel translation along $\g$, thus $\t_t\in\tgtm\otimes\ctxm$ and $\t_0=\id$, $\nabla\t_t=0$.
Let $\bar\mu$ denote the law of $(\t_tz_t)_{t\in[0,1]}$ on $\tgoxy$.
The following result is well known.
See \cite{K} and, for (vi), \cite[Theorem 4.17]{B}.
Although it is framed geometrically, it is essentially a result about quadratic equations in matrices and Gaussian processes,
as one can see by choosing a chart along $\g$ such that $a(\g_t)=\t_t=\id$ for all $t$. 
\begin{proposition}\label{1.5}
The following are equivalent
\begin{itemize}
\item[{\rm (i)}] $J_t$ is invertible for all $t\in(0,1]$,
\item[{\rm (ii)}] $K_t$ is invertible for all $t\in[0,1)$,
\item[{\rm (iii)}] there exists a $C^1$ path $(A_t)_{t\in[0,1)}$ along $\g$, with $A_t\in T_{\g_t}M\otimes T^*_{\g_t}M$,
solving the Riccati equation
$$
\nabla A_t+A_t^2+R_t=0,\quad (1-t)A_t\ra-\id\quad {\text as}\quad t\to 1,
$$
\item[{\rm (iv)}] there exists a $C^1$ path $(B_t)_{t\in(0,1]}$ along $\g$, with $B_t\in T_{\g_t}M\otimes T^*_{\g_t}M$,
solving the Riccati equation
$$
\nabla B_t+B_t^2+R_t=0,\quad tB_t\ra\id\quad {\text as}\quad t\to 0,
$$
\item[{\rm (v)}] $Q$ is positive-definite on $\tghxy$,
\item[{\rm (vi)}] we have
$$
\int_\tgoxy\exp\left\{\frac12\int_0^1\<v_t,R_tv_t\>\, dt\right\}\bar\mu(dv)<\infty.
$$
\end{itemize}
Moreover, under these conditions, we have 
$$
J_1=K_0^*,\q \nabla K_t=A_tK_t,\q \nabla J_t=B_tJ_t
$$
and for $t<1$
$$
J_t=K_t\int_0^tK_s^{-1}a(\g_s)(K_s^{-1})^*\,ds\,K_0^*.
$$
\end{proposition}
It is standard that the condition that $\g$ is minimal implies already that $Q$ is non-negative and that $J_t$ is invertible for all $t\in(0,1)$.
In this context then, condition (i) is equivalent to the condition that $x$ and $y$ are non-conjugate along $\g$, that is,
there is no non-trivial vector field $(v_t)_{t\in[0,1]}$ along $\g$ vanishing at the endpoints and such that $\nabla^2 v_t+R_tv_t=0$ for all $t$.
In Section \ref{SUB}, we saw that (i), (ii) and (v) were also equivalent in the sub-Riemannian case.

The following result gives three further characterizations for the Gaussian measure $\mu_\g$,
which by Theorem \ref{MR} describes the small-time fluctuations of the Brownian bridge in $M$.
As in the preceding result, by choice of a suitable chart along $\g$, we can reduce to the case where $M=\R^d$ with $a(\g_t)=\t_t=\id$ for all $t$. 
The result is then of a standard type for Gaussian processes.
See for example \cite{NV}.
\begin{theorem}\label{RMR}
Let $M$ be a connected Riemannian manifold and let $x,y\in M$.
Suppose that there is a unique minimal path $\g\in\hxy$ and that $(x,y)$ is non-conjugate along $\g$.
Then there exists a zero-mean Gaussian probability measure $\mg$ on $\tgoxy$
with the following properties:
\begin{itemize}
\item[{\rm (i)}] for $s\le t$, we have
$$
\int_{\tgoxy} v_s\otimes v_t\,\mg(dv)=J_sJ_1^{-1}K_t^*,
$$
\item[{\rm (ii)}] for all continuous linear functionals $\phi$ on $\tgoxy$, we have
$$
\int_\tgoxy\phi(v)^2\mg(dv)=Q(\tilde\phi)
$$
where $\tilde\phi\in\tghxy$ is determined by $\phi(v)=Q(\tilde\phi,v)$ for all $v\in\tghxy$,
\item[{\rm (iii)}] under $\mg$, the coordinate process $v$ on $\tgoxy$ satisfies a covariant linear stochastic differential equation over $\g$ of the form
$$
\nabla v_t=\t_tdb_t+A_tv_tdt,\quad v_0=0,
$$
\item[{\rm (iv)}] $\mg$ is absolutely continuous with respect to $\bar\mu$, with Radon--Nikodym derivative 
\begin{equation}\label{exp}
\frac{d\mg}{d\bar\mu}(v)\propto\exp\left\{\frac12\int_0^1\<v_t,R_tv_t\>\, dt\right\}.
\end{equation}
\end{itemize}
Moreover, any one of these properties characterizes $\mg$ uniquely.
\end{theorem}
The equivalence of (i) and (ii) was established in a more general context in Section \ref{QG}.
Here is an illustrative calculation, taking advantage of the reduction to $\t_t=\id$ mentioned above. 
Suppose that $\mu$ satisfies (iii). 
Set $w_t=K_t^{-1}v_t$, then $dw_t=K_t^{-1}db_t$, so
$$
w_t=\int_0^tK_s^{-1}db_s.
$$
Hence, for $0\le s\le t\le 1$,
$$
\int_{\tgoxy} v_s\otimes v_t\,\,\mu(dv)=K_s\left(\int_0^sK_r^{-1}(K_r^{-1})^*dr\right)K_t^*=J_sJ_1^{-1}K_t^*.
$$
Hence $\mu$ satisfies (i).

We examine now how our result specializes in some simple cases. 
When $M=\R^d$ with Euclidean metric, the analysis is trivial, because $\tmexy=\mg$ for all $\ve>0$. 
Then
$$
J_s=sI,\q K_t=(1-t)I,\q R_t=0,\q A_t=-I/(1-t)
$$
so the alternatives in Theorem \ref{RMR} are simply some of the standard
descriptions of the Brownian bridge in $\R^d$.

In the case where $M$ is a sphere or hyperbolic space, we can rewrite \eqref{exp} in the form
$$
\frac{d\mg}{d\bar\mu}(y)
\propto\exp\left\{\frac{ K d(x,y)^2}{2}\int_0^1|y_t|^2\, dt\right\},
$$
where $K$, the scalar curvature, is $1$ for the sphere and $-1$ for hyperbolic space. 
Thus, on a sphere, the variance of the fluctuations is larger than in $\R^d$, whereas, in hyperbolic space it is less.
This does not contradict the tendency of Brownian paths to separate quickly in hyperbolic space because we are conditioning on the endpoint.
Thus we tend to see those paths which have never deviated far from the geodesic.

\bibliography{p}

\end{document}